\documentclass[11pt]{article}
\usepackage[margin=1in]{geometry}  
\usepackage{graphicx}              
\usepackage{amsmath, amssymb, amsthm, epsfig}               
\usepackage{enumerate}
\usepackage{amsfonts}              
\usepackage{amsthm}                
\usepackage{amsthm}
\usepackage{enumerate}
\theoremstyle{plain}
\usepackage{color}
\usepackage{hyperref}
\hypersetup{colorlinks,citecolor=blue}
\usepackage{cite}
\newtheorem{corollary}{Corollary}[section]
\newtheorem{definition}[corollary]{Definition}

\newtheorem{lemma}[corollary]{Lemma}
\newtheorem{prop}[corollary]{Proposition}
\newtheorem{rem}[corollary]{Remark}
\newtheorem{thm}[corollary]{Theorem}

\newfont{\sBlackboard}{msbm10 scaled 1200}

\newcommand{\mylabel}[1]{\label{#1}
    \ifx\undefined\stillediting
    \else \fbox{$#1$}\fi }
\newcommand{\BE}{\begin{equation}}

\newcommand{\EEQ}{\end{equation}}
\newcommand{\rfb}[1]{\mbox{\rm
        (\ref{#1})}\ifx\undefined\stillediting\else:\fbox{$#1$}\fi}

\newfont{\Blackboard}{msbm10 scaled 1200}

\newfont{\roma}{cmr10 scaled 1200}

\newcommand{\bb}{\begin{equation}}
\newcommand{\bbb}{\end{equation}}

\newcommand{\mm}    {{\hbox{\hskip 0.5pt}}}

\newcommand{\bluff} {{\hbox{\raise 15pt \hbox{\mm}}}}
scaled\magstep2

\makeatletter
\def\section{\@startsection {section}{1}{\z@}{-3.5ex plus -1ex minus
        -.2ex}{2.3ex plus .2ex}{\large\bf}}
\makeatother
\numberwithin{equation}{section}

\begin{document}
\title{On the Schr\"odinger-Bopp-Podolsky system: Ground state \& least energy nodal solutions with nonsmooth nonlinearity}
\author{Anouar Bahrouni\footnote{
Anouar.Bahrouni@fsm.rnu.tn;\ bahrounianouar@yahoo.fr}\ \ and \ Hlel
Missaoui\footnote{
hlel.missaoui@fsm.rnu.tn;\ hlelmissaoui55@gmail.com}\\
Mathematics Department, Faculty of Sciences, University of Monastir,\\ 5019 Monastir, Tunisia}
\maketitle


\begin{abstract}
In this paper, the following Schr\"odinger-Bopp-Podolsky system is studied
\begin{equation*}
\left\lbrace
\begin{array}{lll}
   -\triangle u+V(x)u+q^2\phi u&=f(x,u),  \ \ &\text{in}\ \mathbb{R}^3,  \\
   \ & \ & \ \\
   -\triangle\phi+a^2\triangle^2 \phi&=4\pi u^2,\ \ &  \text{in}\ \mathbb{R}^3,
\end{array}
\right.
\end{equation*}
where  $a>0$ and $q\neq 0$. Under suitable assumptions on $f$ and $V$, by using minimizations arguments and generalized subdifferential, the existence of a ground state with a fixed sign  and a least energy nodal solutions for this system are
obtained. Moreover, we prove that the energy of the nodal solution is twice as large as that of the ground state solution.
\end{abstract}



\section{Introduction}
This paper was motivated by some works that have appeared in recent years concerning the  nonlinear Schr\"odinger-Bopp-Podolsky system of the
type
\begin{equation}\label{BPS}
\left\lbrace
\begin{array}{lll}
   -\triangle u+V(x)u+q^2\phi u&=f(x,u),  \ \ &\text{in}\ \mathbb{R}^3,  \\
   \ & \ & \ \\
   -\triangle\phi+a^2\triangle^2 \phi&=4\pi u^2,\ \ &  \text{in}\ \mathbb{R}^3,
\end{array}
\tag{$\mathcal{BPS}$}
\right.
\end{equation}
where $u,\phi:\mathbb{R}^3\rightarrow \mathbb{R}$, $a>0$ is the Bopp-Podolsky parameter, $q\neq 0$, $V:\mathbb{R}^3\rightarrow \mathbb{R}_+$ is the potential function
 and  $f:\mathbb{R}^3\times \mathbb{R}\rightarrow \mathbb{R}$ is a carath\'eodory function.\\

To the best of our knowledge, this new kind of elliptic system \eqref{BPS} was introduced recently for the first time
in \cite{jde} by P. d'Avenia and G. Siciliano, although the problem was known among  physicists. The system \eqref{BPS} was named Schr\"odinger-Bopp-Podolsky system because it appears when we couple a Schr\"odinger field   $\psi(x,t)$ with
its electromagnetic field in the Bopp-Podolsky electromagnetic theory. The
Bopp-Podolsky theory, see \cite{bopp,podl},  is a second-order gauge theory for the electromagnetic field. It was introduced to solve the so-called "infinity problem" that appears in the classical Maxwell theory. In fact, by the well-known Poisson equation (or Gauss law), the electrostatic potential $\phi$ for a given charge distribution whose density is $\rho$ satisfies the equation
\begin{equation}\label{1.2}
    -\triangle \phi=\rho,\ \ \text{on}\  \mathbb{R}^3.
\end{equation}
If $\rho=4\pi\delta_{x_0}$, ($x_0\in \mathbb{R}^3$), then $\mathcal{G}(x-x_0)$, with $\displaystyle{\mathcal{G}(x):=\frac{1}{\vert x\vert}}$, is the fundamental solution
of \eqref{1.2} and $$\displaystyle{\mathcal{E}_M}(\mathcal{G}):=\frac{1}{2}\int_{\mathbb{R}^3}\vert \nabla\mathcal{G} \vert^2dx=+\infty$$
its electrostatic energy. Thus, in the Bopp-Podolsky theory, the   equation \eqref{1.2} is replaced by
\begin{equation}\label{1.1}
-\triangle\phi+a^2\triangle^2 \phi=\rho,\ \ \text{on}\  \mathbb{R}^3.
\end{equation} Therefore, In this case, if $\rho=4\pi\delta_{x_0}$, ( $x_0\in \mathbb{R}^3$), we are able to know explicitly the solution of the  equation \eqref{1.1}
and to see that its energy is finite or not. Fortunately, in \cite{jde}  P. d'Avenia and G. Siciliano proved that
 $\mathcal{K}(x-x_0)$ with, $\displaystyle{\mathcal{K}(x):=\frac{1-e^{-\frac{\vert x\vert}{a}}}{\vert x\vert}}$, is the fundamental solution of the equation
$$-\triangle\phi+a^2\triangle^2 \phi=4\pi\delta_{x_0},\ \ \text{on}\  \mathbb{R}^3.$$
The solution of the previous equation  has no singularity in $x_0$ since it satisfies
$$\lim\limits_{x\rightarrow x_0}\mathcal{K}(x-x_0)=\frac{1}{a},$$
and its energy is
$$\mathcal{E}_{BP}(\mathcal{K}):=\frac{1}{2}\int_{\mathbb{R}^3}\vert \nabla \mathcal{K}\vert^2dx+\frac{a^2}{2}\int_{\mathbb{R}^3}\vert \triangle \mathcal{K}\vert^2dx<+\infty.$$
For more details about this subject see \cite[Section 2]{jde}.\\

An important fact involving system \eqref{BPS} is that this class of system can be transformed into a
Schr\"odinger equation with a nonlocal term (see\cite{jde}), which allows us to use variational
methods. Effectively, in \cite{jde}, P. d'Avenia and G. Siciliano, by the Lax-Milgram Theorem, proved that for a given $u\in H^1(\mathbb{R}^3)$, there exists a
unique $\phi_u\in \mathcal{D}$ such that
$$-\triangle\phi_u+a^2\triangle^2 \phi_u=4\pi u^2,\ \text{in}\ \mathbb{R}^3,$$
where $\mathcal{D}$ is the completion of $C_0^{\infty}(\mathbb{R}^3)$ with respect to the norm $\Vert \cdot\Vert_{\mathcal{D}}$ induced by the scalar product
$$\langle \psi,\varphi\rangle_{\mathcal{D}}:=\int_{\mathbb{R}^3}\nabla \psi\nabla\varphi+a^2\triangle\psi\triangle\varphi dx.$$
Otherwise, $$\mathcal{D}:=\left\lbrace \phi\in \mathcal{D}^{1,2}(\mathbb{R}^3):\ \triangle \phi\in L^2(\mathbb{R}^3)\right\rbrace,$$
with
$$\mathcal{D}^{1,2}(\mathbb{R}^3):=\left\lbrace \phi\in L^{6}(\mathbb{R}^3):\ \nabla \phi\in L^2(\mathbb{R}^3)\right\rbrace.$$
The space $\mathcal{D}$ is an Hilbert space continuously embedded into $\mathcal{D}^{1,2}(\mathbb{R}^3)$, $L^6(\mathbb{R}^3)$ and $L^\infty (\mathbb{R}^3).$

Moreover, they proved that the unique solution $\displaystyle{\phi_u:=\mathcal{K}*u^2=\frac{1-e^{-\frac{\vert x\vert}{a}}}{\vert x\vert}*u^2}$, for all $u\in H^1(\mathbb{R}^3)$  verifies the following properties
\begin{lemma}[see \cite{jde}]\label{lem1}\ \\
For any $u\in H^1(\mathbb{R}^3)$, we have:
\begin{enumerate}
    \item[$(1)$] for every $y\in \mathbb{R}^3$, $\phi_{u(\cdot+y)}=\phi_{u}(\cdot+y)$;
    \item[$(2)$] $\phi_u\geq 0;$
    \item[$(3)$] for every $r\in(3,+\infty]$,\ $\phi_u\in L^r(\mathbb{R}^3)\cap C_0(\mathbb{R}^3)$;
    \item[$(4)$]for every $r\in(\frac{3}{2},+\infty]$, $\nabla\phi_u=\nabla\left(\frac{1-e^{-\frac{\vert x\vert}{a}}}{\vert x\vert}\right)*u^2\in L^{r}(\mathbb{R}^3)\cap C_0(\mathbb{R}^3)$;
    \item[$(5)$]$\phi_u\in\mathcal{D}$;
    \item[$(6)$]$\Vert \phi_u\Vert_6\leq C\Vert u\Vert^2$;
    \item[$(7)$] $\phi_u$ is the unique minimizer of the functional
    $$E(\phi):=\frac{1}{2}\Vert \nabla \phi\Vert^2_2+\frac{a^2}{2}\Vert \triangle \phi\Vert^2_2-\int_{\mathbb{R}^3}\phi u^2dx,\ \ \text{for all}\ \phi \in\mathcal{D};$$
    \item[$(8)$]if $v_n\rightharpoonup v$ in $H^1(\mathbb{R}^3)$, then $\phi_{v_n}\rightharpoonup \phi_v$ in $\mathcal{D}$ and $\displaystyle{\int_{\mathbb{R}^3}\phi_{u_n}u_n^2 dx\rightarrow \int_{\mathbb{R}^3}\phi_{u}u^2 dx };$
   \item[$(9)$] $\phi_{tu}=t^2\phi_u$, for all $t\in \mathbb{R}_+$;
   \item[$(10)$]$\displaystyle{\int_{\mathbb{R}^3}\phi_u u^2dx=\int_{\mathbb{R}^3}\int_{\mathbb{R}^3}\frac{1-e^{-\frac{\vert x-y\vert}{a}}
   }{\vert x-y\vert}u(x)^2u(y)^2dxdy\leq \frac{1}{a}\Vert u\Vert_2^4}.$
\end{enumerate}
\end{lemma}

Therefore, the pair $(u,\phi)\in H^1_V(\mathbb{R}^3)\times \mathcal{D}$ ( where $H^1_V(\mathbb{R}^3)$ is a subspace from $H^1(\mathbb{R}^3)$ and it will be defined later in Sect. 2 ) is a solution of \eqref{BPS} if, and only if, $\phi=\phi_u$ and $u\in H^1_V(\mathbb{R}^3)$ is a weak solution of the nonlocal problem
\begin{equation}\label{P}
   -\triangle u+V(x)u+q^2\phi_u u=f(x,u),  \ \ \text{in}\ \mathbb{R}^3. \tag{$\mathcal{P}$}
\end{equation}

By these facts, in the last years, many authors that studied the
system \eqref{BPS} have focused their attention on the problem
\eqref{P} aiming to establish the existence and nonexistence of
solutions, the multiplicity of solutions, ground state solutions,
radial, and nonradial solutions, see
\cite{AS,jde,AG,CRT,CT,FS,EH,HE,EH1,LC,LPT,LT1,MS,GK,TY,YCL,Z1,ZCC}.
The paper \cite{jde} seems to be the first paper that Schr\"
odinger-Bopp-Podolsky system is studied in the mathematical
literature. In the said paper, d'Avenia and Siciliano consider the
following Schr\"odinger-Bopp-Podolsky system:
\begin{equation}
\left\lbrace
\begin{array}{lll}
   -\triangle u+w u+q^2\phi u&=|u|^{p-2}u,  \ \ &\text{in}\ \mathbb{R}^3,  \\
   \ & \ & \ \\
   -\triangle\phi+a^2\triangle^2 \phi&=4\pi u^2,\ \ &  \text{in}\ \mathbb{R}^3,
\end{array}
 \right.
\end{equation}
where $a,w>0$, $q\neq 0$, $p\in (2,6)$. They studied the existence,
nonexistence and the behavior of the solution as $a\rightarrow 0$.
Again the solutions converge to the solution of the limit case.
Moreover, by using a Pohozaev-type identity, they proved that the
above system does not admit any nontrivial solution when $p \geq 6$.
However, the authors do not cover critical cases. In \cite{CT}, Chen
and Tang studied the existence of solutions of the critical system
\eqref{BPS}. By using some new analytic techniques and new
inequalities, they found nontrivial solutions, ground state
solutions of Nehari-Pohozaev type and ground state solutions of
Nehari type in constant potential case.  In all those papers, the
solutions found are nonnegative or with unknown signs. However,
related to nodal ( or sign-changing ) solutions, we found a few
papers, see \cite{HWT,WCL,Z2}. In \cite{HWT}, the authors proved the
existence of least-energy sign-changing solutions for the critical
Schr\"odinger-Bopp-Podolsky system. In \cite{WCL},  by using a
perturbation approach and the method of invariant sets of descending
flow incorporated with minimax arguments, the authors proved the
existence and multiplicity of sign-changing solutions for problem
\eqref{BPS} with superlinear nonlinearity. Moreover, the asymptotic
behavior of sign-changing solutions is also established. In
\cite{Z2}, Q. Zhang used some stronger conditions on the
nonlinearity reaction term  as the well-known
Ambrosetti–Rabinowitz (AR for short). \\

 Motivated by the above papers, in this work we are interested in finding ground state and least-energy nodal solutions for the problem \eqref{BPS}, or rather for the
 system \eqref{P}, without the (AR) or differentiability conditions on the non-linear term $f$. More precisely,  we assume the following assumption on the potential function $V(\cdot)$:
\begin{enumerate}
\item[$(V_0)$]$V(x) \in C(\mathbb{R}^3,\mathbb{R}_+)$  is coercive.
\end{enumerate}
For what concerns the nonlinearity reaction  term $f:\mathbb{R}^3\times \mathbb{R}\rightarrow \mathbb{R}$, we assume that $f$ is a measurable, locally Lipschitz in the second variable $t\in \mathbb{R}$, and satisfies the following assumptions:
\begin{enumerate}
\item[$(f_1)$] $\displaystyle{\lim\limits_{t\rightarrow 0}\frac{f(x,t)}{t}=0}$ uniformly in $x\in \mathbb{R}^3$;
\item[$(f_2)$] $\displaystyle{\lim\limits_{\vert t\vert\rightarrow +\infty}\frac{f(x,t)}{t^5}=0}$ uniformly in $x\in \mathbb{R}^3$;
\item[$(f_3)$]  $\displaystyle{\lim\limits_{t\rightarrow \pm\infty}\frac{F(x,t)}{t^4}=+\infty}$ uniformly in $x\in \mathbb{R}^3$, where $\displaystyle{F(x,t)=\int_0^t f(x,s)ds}$;
\item[$(f_4)$] $0<3f(x,t)t\leq f^*(x,t)t^2$, for all $f^*(x,t)\in \partial f(x,t)$, a.a. $x\in \mathbb{R}^3$, and all $\vert t\vert>0$, where $\partial f(x,t)$ is the set of the "generalized
subdifferential" of $f(x,\cdot)$ at $t$ ( which will be defined later in Sect. 2 ).
\end{enumerate}

Once that we will apply variational
methods on the problem \eqref{P}, we find  the term $\displaystyle{\int_{\mathbb{R}^3}\phi_u u^2dx}$ which is homogeneous of degree 4. Thus, the corresponding Ambrosetti-Rabinowitz
condition on $f(\cdot,\cdot)$ is the following:
\begin{enumerate}
    \item[(AR)] There exists $\theta>4$ such that $0<\theta F(x,t)\leq f(x,t)t$, for a.a. $x\in \mathbb{R}^3$ and all $\vert t\vert >0$.
\end{enumerate}
Therefore, our assumption $(f_4)$ is weaker than the (AR) condition (see Remark \ref{4rem1}).\\

 Next, we state our main results for problem \eqref{P}.

\begin{thm}\label{thm1}
Suppose that the hypotheses $(f_1)-(f_4)$ and $(V_0)$ hold. Then, problem \eqref{P} has a ground state solution $\widehat{u}\in H^1_V(\mathbb{R}^3)$.
\end{thm}
\begin{thm}\label{thm2}
Suppose that the hypotheses $(f_1)-(f_4)$ and $(V_0)$ hold. Then,
problem \eqref{P} has a least energy sign-changing solution $\widehat{w}\in H^1_V(\mathbb{R}^3)$. Moreover, the energy of the solution $\widehat{w}$  is twice as large as
that of the ground state solution $\widehat{u}$.
\end{thm}
\begin{corollary}
 As a consequence of Theorems \ref{thm1} and \ref{thm2}, the ground state solution $\widehat{u}$ of problem \eqref{P} is with a fixed sign.
\end{corollary}

 The paper is organized as follows. In Sect. 2, we give the corresponding energy functional to problem \eqref{P} and some tools. The Sect. 3 is devoted to proving the existence of ground state solution to problem \eqref{P}. Finally, we show the existence of a least energy nodal solution to problem \eqref{P} and we compare the energy levels of the ground state and nodal solution.
\section{Preliminary and tools}
 In this Section, we give the variational setting and the corresponding energy functional associated with problem \eqref{P}, and the definition of "generalized subdifferential" which are needed in the sequel.

 First, denote by $L^{r}(\mathbb{R}^3)$, for all $r\in [1,+\infty)$, the usual Lebesgue space with norm
 $$\Vert u\Vert_r:=\left( \int_{\mathbb{R}^3}\vert u(x)\vert^rdx\right)^\frac{1}{r}.
 $$
Under the assumption $(V_0)$, we could define the Hilbert space
$$H_V^1(\mathbb{R}^3):=\left\lbrace u\in H^1(\mathbb{R}^3):\ \int_{\mathbb{R}^3}V(x) u^2dx<+\infty\right\rbrace,$$
 which is equipped with the norm
 $$\Vert u\Vert:=\left( \int_{\mathbb{R}^3}\vert \nabla u\vert^2dx+\int_{\mathbb{R}^3}V(x) u^2dx\right)^\frac{1}{2},\ \ \text{for all}\ u\in H_V^1(\mathbb{R}^3).$$
 \begin{thm}[see \cite{bart}]\label{thm3}
  Under the assumption $(V_0)$, the space $H_V^1(\mathbb{R}^3)$ is compactly embedded into $L^r(\mathbb{R}^3)$, for all $r\in[2,6)$. Moreover,  the embedding $H_V^1(\mathbb{R}^3)\hookrightarrow L^{s}(\mathbb{R}^3)$, for all $s\in[2,6]$ is continuous, then there exists $C_s>0$ such that
  $$\Vert u\Vert_s\leq C_s \Vert u\Vert,\ \text{for all}\ u\in H_V^1(\mathbb{R}^3).$$
 \end{thm}
For each $u\in H^1_V(\mathbb{R}^3)$, we denote by $u^-:=\min\lbrace 0,u\rbrace$ and $u^+:=\max\lbrace 0,u\rbrace$. As a consequence, from Theorem \ref{thm3} and assumptions $(f_1)-(f_2)$, we have the following lemma.
\begin{lemma}\label{9lem1}
Assume that assumptions $(f_1)-(f_2)$ and $(V_0)$ hold. Let $\lbrace u_n\rbrace_{n\in \mathbb{N}}\subset H^1_V (\mathbb{R}^3)$ such that $u_n \rightharpoonup u$ in $H^1_V(\mathbb{R}^3)$ as $n\rightarrow +\infty$, then
\begin{enumerate}
    \item[$(1)$] $\displaystyle{\lim\limits_{n\rightarrow +\infty}\int_{\mathbb{R}^3}F(x,u_n)dx=\int_{\mathbb{R}^3}F(x,u)dx}.$
    \item[$(2)$]$\displaystyle{\lim\limits_{n\rightarrow +\infty}\int_{\mathbb{R}^3}F(x,u_n^\pm)dx=\int_{\mathbb{R}^3}F(x,u^\pm)dx}.$
    \item[$(3)$]$\displaystyle{\lim\limits_{n\rightarrow +\infty}\int_{\mathbb{R}^3}f(x,u_n)u_ndx=\int_{\mathbb{R}^3}f(x,u)udx}.$
    \item[$(4)$]$\displaystyle{\lim\limits_{n\rightarrow +\infty}\int_{\mathbb{R}^3}f(x,u_n^\pm)u_n^\pm dx=\int_{\mathbb{R}^3}f(x,u^\pm)u^\pm dx}.$
\end{enumerate}
\end{lemma}
Next, we would like to mention that, from assumptions $(f_1)-(f_2)$  and Lemma \ref{lem1}, the existence of solutions for problem \eqref{P} can be made via
variational methods. In particular,
 the corresponding energy functional to problem \eqref{P} is $J:H^1_V(\mathbb{R}^3)\longrightarrow \mathbb{R}$, which is defined by
\begin{align}\label{J}
J(u)&:=\frac{1}{2} \left(\int_{\mathbb{R}^3}\vert\nabla u\vert^2dx+\int_{\mathbb{R}^3}V(x) u^2dx\right)+\frac{q^2}{4}\int_{\mathbb{R}^3}\phi_u u^2dx-\int_{\mathbb{R}^3}F(x,u)dx\nonumber\\
&=\frac{1}{2} \Vert u\Vert^2+\frac{q^2}{4}\int_{\mathbb{R}^3}\phi_u u^2dx-\int_{\mathbb{R}^3}F(x,u)dx,
\end{align}
 belongs to $C^1\left(H^1_V(\mathbb{R}^3),\mathbb{R}\right)$ and has the following derivative
$$J^{'}(u)v=\int_{\mathbb{R}^3}\nabla u\nabla v dx +\int_{\mathbb{R}^3} V(x)uvdx+q^2\int_{\mathbb{R}^3} \phi_u uvdx-\int_{\mathbb{R}^3}f(x,u)vdx,\ \text{for all}\ u,v\in H^1_V(\mathbb{R}^3).$$
Hence, critical points of $J$ are the weak solutions for the nonlocal problem \eqref{P}.\\

In what follows, we denote by $\mathcal{N}$  the Nehari manifold associated with $J$, that is,
$$\mathcal{N}:=\left\lbrace u\in H^1_V(\mathbb{R}^3):\  u\not\equiv 0,\ \langle J^{'}(u),u\rangle=0\right\rbrace,$$
where $\langle \cdot,\cdot\rangle$ is the duality brackets between $H^1_V(\mathbb{R}^3)$ and it topological dual $\left(H^1_V(\mathbb{R}^3)\right)^*$.\\
A critical point $u\not\equiv 0$ of $J$ is a ground state of \eqref{P} if
$$J(u):=\inf_{v\in \mathcal{N}}J(v).$$
Since we are looking for least energy nodal solutions (or sign-changing solutions), our second goal is to prove the
existence of a critical point for $J$ in the set
$$\mathcal{M}:=\left\lbrace u\in\mathcal{N}:\  u^\pm\not\equiv 0,\ \langle J^{'}(u),u^+\rangle=0=\langle J^{'}(u),u^-\rangle\right\rbrace.$$

A main tool used in the present paper is the subdifferential theory of Clark \cite{418,419} for locally Lipschitz functionals. Let $\mathbb{X}$ be a Banach space, $\mathbb{X}^*$ its topological dual, and let $\langle \cdot ,\cdot\rangle_{\mathbb{X}}$ denote the duality brackets for the pair $(\mathbb{X},\mathbb{X}^*)$.
\begin{definition}
Let the functional $\Psi:\mathbb{X}\rightarrow \mathbb{R}$. We say that $\Psi$ is locally Lipschitz if, for every $x\in \mathbb{X}$, there exists an open neighborhood $U(x)$ of $x$ and $k_x>0$ such that
$$\vert \Psi(u)-\Psi(v)\vert\leq k_x \Vert u-v\Vert_{\mathbb{X}}\ \ \text{for all}\ u,v\in U(x).$$
\end{definition}
\begin{definition}
Given a locally Lipschitz function $\Psi:\mathbb{X}\rightarrow \mathbb{R}$, the "generalized directional derivative" of $\Psi$ at $u\in\mathbb{X}$ in the direction $v\in \mathbb{X}$, denoted by $\tilde{\Psi}(u;v)$, is defined by
$$\tilde{\Psi}(u;v)=\mathop{\limsup_{x\rightarrow u}}_{t\searrow 0}\frac{\Psi(x+tv)-\Psi(x)}{t}.$$
\end{definition}
\begin{definition}
The "generalized subdifferential" of $\Psi$ at $u\in\mathbb{X}$ is the set $\partial\Psi(u)\subseteq \mathbb{X}^*$ given by
$$\partial\Psi(u):=\left\lbrace \Psi^*\in \mathbb{X}^*:\langle \Psi^*,v\rangle_{\mathbb{X}}\leq \tilde{\Psi}(u;v),\ \ \text{for all}\ v\in \mathbb{X}\right\rbrace.$$
\end{definition}
The Hahn-Banach theorem implies that $\partial \Psi (u)\neq \emptyset$ for all $u\in\mathbb{X}$, it is convex and $w^*$-compact ( in  weak topology sense ). If $\Psi$ is also convex, then it coincides with the subdifferential in the sense of convex functionals, see \cite{420}. If $\Psi\in C^{1}(\mathbb{X},\mathbb{R})$, then $\partial \Psi (u)=\left\lbrace \Psi ^{'}(u)\right\rbrace$. Note that the generalized subdifferential has a remarkable calculus, similar to that in the classical derivative, see \cite{418,419,420}.
\begin{rem}\label{4rem1}
$\bullet$ The assumption $(f_4)$ and  the generalized subdifferential calculus of Clarke \cite[ p. 48]{419}, give that  for a.a. $x\in\mathbb{R}^3$, we have
\begin{equation}\label{1}
t\longmapsto \frac{f(x,t)}{\vert t\vert^{3}}\ \ \text{is nondecreasing in}\ \vert t\vert >0
\end{equation}
and
\begin{equation}\label{2}
t\longmapsto f(x,t)t-4F(x,t)\ \ \text{is nondecreasing on}\ \mathbb{R}_+\ \text{and nonincreasing on}\ \mathbb{R}_-.
\end{equation}
$\bullet$ The assumption $(f_3)$ is weaker than the (AR)  condition. Indeed, the  function  $$\displaystyle{f(s)=\vert s\vert^{2}s\ln(1+\vert s\vert)}$$ $($for the sake of simplicity, we drop the $x$-dependence$)$ satisfies hypotheses $(f_1)-(f_4)$ but not the $(AR)$ condition.\\
$\bullet$ By hypothesis $(f_4)$ and the fact that $f(x,0)=0$, for a.a. $x\in \mathbb{R}^3$, we have
\begin{equation}\label{3}
  f(x,t)\geq 0\ (\leq 0),\  \text{for a.a.}\ x\in \mathbb{R}^3\ \text{and all}\ t\geq 0\ (t\leq 0).
\end{equation}
  Therefore,
  \begin{equation}\label{4}
  F(x,t)=\int_0^t f(x,s)ds \geq 0,\ \text{for a.a.}\ x\in \mathbb{R}^3\ \text{and all}\ t\geq 0.
  \end{equation}
  On the other side, if $t<0$, by \eqref{1} and \eqref{3}, for a.a. $x\in \mathbb{R}^3$, we have
\begin{align*}
  F(x,t)&=\int_0^t f(x,s)ds=\int_0^t \frac{f(x,s)}{\vert s\vert^{3}}\vert s\vert^{3} ds\geq \frac{f(x,t)}{\vert t\vert^{3}}\int_0^t \vert s\vert^{3} ds=\frac{1}{4}f(x,t)t \geq 0.
\end{align*}
It follows, by \eqref{4}, that
\begin{equation}\label{5}
F(x,t)\geq 0,\ \text{for a.a.}\ x\in \mathbb{R}^3\ \text{and all}\ t\in \mathbb{R}.
\end{equation}
\end{rem}

\section{Ground state solution}
In this section, we prove the existence of a weak solution of \eqref{P} which minimizes $J_{\left \vert_ \mathcal{N}\right.}$. Such a solution is known as
a "ground state solution".
\begin{lemma}\label{lem3}
Assume that the assumptions  $(f_1)-(f_4)$ and $(V_0)$ hold. Then, for each $u\in H^1_V(\mathbb{R}^3)\setminus\lbrace 0\rbrace$, there exists a unique $t_u>0$ such that $t_u\in\mathcal{N}$.
\end{lemma}
\begin{proof} Let $u\in H^1_V(\mathbb{R}^3)\setminus\lbrace 0\rbrace$, we define the fibering map $h_u:(0,+\infty)\longrightarrow \mathbb{R}$ by
\begin{equation}\label{5eq2}
    h_u(t):=J(tu),\ \ \text{for all}\ t>0.
    \end{equation}
 From Lemma \ref{lem1}, for all  $t\in (0,+\infty)$, we have
\begin{align}
h_u(t)& =J(tu)=\frac{1}{2}\Vert tu\Vert^2+\frac{q^2}{4}\int_{\mathbb{R}^3}\phi_{tu}(x) (tu)^2dx-\int_{\mathbb{R}^3}F(x,tu)dx\nonumber\\
& =\frac{1}{2}t^2\Vert  u\Vert^2+\frac{q^2}{4}t^4\int_{\mathbb{R}^3}\phi_{u} u^2dx-\int_{\mathbb{R}^3}F(x,tu)dx.
\end{align}
It's clear that $h_u\in C^1\left((0,+\infty),\mathbb{R}\right)$ and
\begin{equation}\label{eqc}
h_u^{'}(t)=t\Vert  u\Vert^2+q^2t^3\int_{\mathbb{R}^3}\phi_{u} u^2dx-\int_{\mathbb{R}^3}f(x,tu)udx.
\end{equation}
Evidently,
$$h_u^{'}(t)=0\Longleftrightarrow tu\in \mathcal{N}.$$
So, the equation $h_u^{'}(t)=0$ is equivalent to
\begin{equation}\label{5eq3}
0<q^2\int_{\mathbb{R}^3}\phi_{u} u^2dx=\int_{\mathbb{R}^3}\frac{f(x,tu)}{ t^3} udx-\frac{1}{t^2}\Vert  u\Vert^2.
\end{equation}
Obviously, that from  \eqref{1} the right hand side of equation \eqref{5eq3} is increasing.\\
Therefore,  there exists a unique $t_u>0$ such that
\begin{equation}\label{eqa}
h_u^{'}(t_u)=0,
\end{equation}
thus,
$$J^{'}(t_uu)u=0\Rightarrow J^{'}(t_uu)t_uu=0 \Rightarrow t_uu\in \mathcal{N}.$$
This ends the proof.
\end{proof}

\begin{lemma}\label{6lem4}
Assume that  assumptions  $(f_1)-(f_4)$  and $(V_0)$ hold. Then, for each $u\in \mathcal{N}$
$$J(tu)\leq J(u),\ \ \text{for all}\ \ t>0.$$
\end{lemma}
\begin{proof}
We consider the fibering map $h_u(\cdot)$ introduced in the proof of Lemma \ref{lem3}. Since $u\in \mathcal{N}$, by  Lemma \ref{lem3}  we have $h_u^{'}(1)=0$ and $t=1$ is the unique critical point of $h_u(\cdot)$.

From assumptions $(f_1)-(f_2)$, we infer that for all $\varepsilon>0$, there is $C_\varepsilon>0$ such that
\begin{equation}\label{5eq1}
\vert f(x,t)\vert\leq \varepsilon\vert t\vert +C_\varepsilon\vert t\vert^5\ \ \text{and}\ \ \vert F(x,t)\vert\leq \varepsilon\vert t\vert^2 +C_\varepsilon\vert t\vert^6,\ \ \text{for all}\ \ t\in \mathbb{R}, \ \text{and all}\ x\in \mathbb{R}^3.
\end{equation}
By \eqref{5eq2}, \eqref{5eq1} and Theorem \ref{thm3}, we get
\begin{align}\label{eqd}
   h_u(t)& =\frac{1}{2}t^2\Vert  u\Vert^2+t^4\frac{q^2}{4}\int_{\mathbb{R}^3}\phi_{u} u^2dx-\int_{\mathbb{R}^3}F(x,tu)dx\nonumber\\
   & \geq \frac{1}{2}t^2\Vert  u\Vert^2+t^4\frac{q^2}{4}\int_{\mathbb{R}^3}\phi_{u} u^2dx-\varepsilon t^2\int_{\mathbb{R}^3}\vert u\vert^2dx-C_\varepsilon t^6\int_{\mathbb{R}^3}\vert u\vert^6dx\nonumber\\
   & \geq \frac{1}{2}t^2\Vert  u\Vert^2-\varepsilon t^2\Vert u\Vert^2_2-C_\varepsilon t^6\Vert u\Vert_6^6\nonumber\\
   & \geq \frac{1}{2}t^2\Vert  u\Vert^2-\varepsilon C_1 t^2\Vert u\Vert^2-C_\varepsilon C_2 t^6\Vert u\Vert^6\nonumber\\
   & \geq \left(\frac{1}{2}-\varepsilon C_1\right)t^2\Vert  u\Vert^2-C_\varepsilon C_2 t^6\Vert u\Vert^6.
\end{align}
Choosing $\varepsilon=\frac{1}{4C_1}$, we find that
\begin{equation}\label{5eq5}
 h_u(t)=J(tu)>0,\ \ \text{for all}\ t\in (0,1)\ \text{small enough}.
\end{equation}
In light of the definition of $h_u(\cdot)$ and \eqref{5}, we see that
\begin{align*}
     \frac{h_u(t)}{t^4}& \leq \frac{1}{2t^2}\Vert  u\Vert^2+\frac{q^2}{4}\int_{\mathbb{R}^3}\phi_{u} u^2dx-\int_{A}\frac{F(x,tu)}{ \vert tu\vert^4}\vert u\vert^4dx,\\
\end{align*}
where $A \subset\text{Supp}(u)$ is a measurable  set with  positive measure. Hence, using Fatou's lemma and assumption $(f_3)$, we deduce that
\begin{equation}\label{eqb}
\limsup_{t\rightarrow +\infty}\frac{h_u(t)}{t^4}\leq \frac{q^2}{4}\int_{\mathbb{R}^3}\phi_{u} u^2dx- \liminf_{t\rightarrow +\infty}\int_{A}\frac{F(x,tu)}{ \vert tu\vert^4}\vert u\vert^4dx=-\infty.
\end{equation}
It follows, by \eqref{5eq5} and the continuity of $h_u(\cdot)$, that $h_u(\cdot)$ has a global maximum point $t_0\in (0,+\infty)$. Therefore, $t_0$ is a critical point of $h_u(\cdot)$. Thus, $t_0=1$  and
$$h_u(t)\leq h_u(1),\ \ \text{for all}\ t\in (0,+\infty).$$
This gives the proof.
\end{proof}
\begin{rem}\label{rem11}
Let $u\in H^1_V(\mathbb{R}^3)\setminus\lbrace 0\rbrace$ and $h_u(\cdot)$ defined as in the proof of Lemma \ref{lem3}.
 Arguing as in \eqref{5eq5} and \eqref{eqb}, we get
$$h_u(t)>0,\ \ \text{for all}\ t>0\ \text{small enough}$$
and
$$\lim\limits_{t\rightarrow +\infty}h_u(t)=-\infty.$$
These facts combined with \eqref{eqa}, give that
$$h_u(t)\ \text{increasing in}\ (0,t_u)\ \ \text{and}\ \ h_u(t)\ \text{decreasing in}\ (t_u,+\infty).$$
Thus, the unique $t_u>0$ in Lemma \ref{lem3} verifies
$$h^{'}_u(t)>0\ \text{in}\ (0,t_u)\ \text{and}\ h^{'}_u(t)<0\ \text{in}\ (t_u,+\infty).$$
\end{rem}
Next, let
$$c_0:=\inf_{u\in \mathcal{N}}J(u).$$
\begin{lemma}\label{lem6}
Assume that assumptions $(f_1)-(f_4)$ and $(V_0)$ hold. Then, $c_0>0$.
\end{lemma}
\begin{proof}
Let $u\in H^1_V(\mathbb{R}^3)\setminus\lbrace 0\rbrace$. By  \eqref{5eq1} and Theorem \ref{thm3}, we get
\begin{align*}
   J(u)& =\frac{1}{2}\Vert  u\Vert^2+\frac{q^2}{4}\int_{\mathbb{R}^3}\phi_{u} u^2dx-\int_{\mathbb{R}^3}F(x,u)dx\\
   & \geq \frac{1}{2}\Vert  u\Vert^2-\varepsilon \int_{\mathbb{R}^3}\vert u\vert^2dx-C_\varepsilon \int_{\mathbb{R}^3}\vert u\vert^6dx\\
   & =\frac{1}{2}\Vert  u\Vert^2-\varepsilon \Vert  u\Vert_2^2-C_\varepsilon \Vert  u\Vert_6^6\\
   & \geq \frac{1}{2}\Vert  u\Vert^2-\varepsilon C_1\Vert  u\Vert^2-C_\varepsilon C_2\Vert  u\Vert^6\\
   & = \left(\frac{1}{2}-\varepsilon C_1\right)\Vert  u\Vert^2-C_\varepsilon C_2\Vert  u\Vert^6.
\end{align*}
Choosing $\varepsilon=\frac{1}{4C_1}$, we infer that
$$J(u)\geq \frac{1}{4} \Vert  u\Vert^2-C_3 \Vert  u\Vert^6.$$
Therefore, there exist $\rho\in (0,1)$  small enough  and $\eta>0$  such that
$$J(u)\geq \eta ,\ \ \text{for all}\ \ \Vert u\Vert=\rho.$$ Let $u\in \mathcal{N}$, choosing $\tau>0$ such that $\Vert \tau u\Vert=\rho$. Exploiting Lemma \ref{6lem4}, we deduce that
$$J(u)\geq J(\tau u)\geq \eta>0.$$
This completes the proof.
\end{proof}
\begin{lemma}\label{lem5+}
Assume that assumptions $(f_1)-(f_4)$  and $(V_0)$ hold. Let $\lbrace u_n\rbrace_{n\in\mathbb{N}}\subset \mathcal{N}$ such that $u_n\rightharpoonup u$ in $H_V^{1}(\mathbb{R}^3)$, then $u\not\equiv 0.$
\end{lemma}
\begin{proof}
First let us observe that, there is $\tau>0$ such that
\begin{equation}\label{7eq88+}
    \Vert v\Vert\geq \tau,\ \ \text{for all}\ v\in \mathcal{N}.
\end{equation}
Indeed: Let $v\in\mathcal{N}$. Then,
\begin{align*}
  \int_{\mathbb{R}^3}\vert\nabla v\vert^2 dx +\int_{\mathbb{R}^3} V(x)v^2dx+q^2\int_{\mathbb{R}^3} \phi_v v^2dx=\int_{\mathbb{R}^3}f(x,v)vdx.
\end{align*}
It follows, by Lemma \ref{lem1}, that
\begin{align}\label{8eq99+}
    \Vert v\Vert^2\leq \int_{\mathbb{R}^3}f(x,v)vdx.
\end{align}
On the other side, from assumptions $(f_1)-(f_2)$, for all $\varepsilon>0$ there is a constant $C_\varepsilon > 0$ such that
\begin{equation}\label{8eq999}
 \vert f(x,t)t\vert\leq \varepsilon \vert t\vert^{2}+C_\varepsilon  \vert t\vert^{6},\ \ \text{for a.a.}\ x\in \mathbb{R}^3\ \text{and all}\ t\in \mathbb{R}.
\end{equation}
Using Theorem \ref{thm1}, \eqref{8eq99+} and \eqref{8eq999}, we find that
\begin{align*}
   \Vert v\Vert^2 &\leq \varepsilon\int_{\mathbb{R}^3} \vert v\vert^2dx+C_\varepsilon\int_{\mathbb{R}^3} \vert v\vert^6dx\\
    & \leq \varepsilon C_1 \Vert v\Vert^2+ C_\varepsilon C_2 \Vert v\Vert^6,\ \text{for some constants}\ C_1,C_2>0.
\end{align*}
 Choosing $\varepsilon=\frac{1}{2C_1}$ in the previous inequality, we get
$$\frac{1}{2}\Vert v\Vert^2\leq C_\varepsilon C_2 \Vert v\Vert^6.$$
Thus,
\begin{equation}\label{eq7715+}
    \tau\leq \Vert v\Vert,\ \text{where}\ \tau=\left(\frac{1}{2C_{\varepsilon}C_2}\right)^{\frac{1}{4}}.
 \end{equation}
Since $\lbrace u_n\rbrace_{n\in\mathbb{N}}\subset \mathcal{N}$, by \eqref{8eq99+} and \eqref{eq7715+}, we have
$$\tau^2\leq \int_{\mathbb{R}^3}f(x,u_n)u_n dx,\ \ \text{for all}\ \  n\in \mathbb{N}.$$
Hence, passing to the limit as $n\rightarrow +\infty$ in the
previous inequality and using Lemma \ref{9lem1}, we see that
$$\tau^{2}\leq \int_{\mathbb{R}^3}f(x,u)u dx.$$
Therefore, $u\not\equiv 0$. This completes the proof.
\end{proof}
In the following proposition, we prove that the infimum of $J$ is attained on $\mathcal{N}$.

\begin{prop}\label{4prop5}
Assume that assumptions $(f_1)-(f_4)$ and $(V_0)$ are satisfied. Then, there exists $\widehat{u}\in\mathcal{N}$ such that $J(\widehat{u})=c_0$.
\end{prop}
\begin{proof}
Let $\lbrace u_n\rbrace_{n\in\mathbb{N}}\subset \mathcal{N}$ such that
\begin{equation}\label{4eq85}
J(u_n)\longrightarrow c_0\ \ \text{as}\ \ n\longrightarrow +\infty.
\end{equation}
\textbf{Claim:} The sequence $\lbrace u_n\rbrace_{n\in\mathbb{N}}$ is bounded in $H^{1}_V(\mathbb{R}^3).$
Indeed, we argue by contradiction, assume that there exists a subsequence, denoted again by $\lbrace u_n\rbrace_{n\in\mathbb{N}}$ such that
$$\Vert u_n\Vert \longrightarrow +\infty\ \text{as}\ n\longrightarrow +\infty.$$
Let
\begin{equation}\label{4eq300}
v_n:=\frac{u_n}{\Vert u_n\Vert},\ \ \text{for all}\ n\in\mathbb{N}.
\end{equation}
It follows that the sequence $\lbrace v_n\rbrace_{n\in\mathbb{N}}$ is bounded in $H^{1}_V(\mathbb{R}^3)$. Thus, up to a subsequence still denoted by $v_n$, there exists $v\in H^{1}_V(\mathbb{R}^3)$ such that
\begin{equation}\label{4eq55}
v_n\rightharpoonup v \ \text{in}\ H^{1}_V(\mathbb{R}^3),\ \ \text{and}\ \ v_n(x)\rightarrow v(x)\ \text{as}\ n\rightarrow+\infty,\ \text{for a.a. in}\ \mathbb{R}^{3}.
\end{equation}
 Using Lemmas \ref{lem1}, \ref{6lem4}, and the fact that  $\lbrace u_n\rbrace_{n\in\mathbb{N}} \in \mathcal{N}$, for all $t\geq 0$, we have
\begin{align}\label{4eq9}
J(u_n) & = J(\Vert u_n\Vert v_n )\geq J(t v_n)\nonumber\\
& = \frac{1}{2} \Vert tv_n\Vert^2+\frac{q^2}{4}\int_{\mathbb{R}^3}\phi_{tv_n} (x)(tv_n)^2dx-\int_{\mathbb{R}^3}F(x,tv_n)dx\nonumber\\
& \geq  \frac{1}{2} t^2  -\int_{\mathbb{R}^{3}}F(x,tv_n) dx.
\end{align}
Assume that $v_n\rightharpoonup v\equiv0$,  by Lemma \ref{9lem1}, we see that
\begin{equation}\label{8eq1}
\int_{\mathbb{R}^{3}}F(x,tv_n) dx\longrightarrow 0,\ \text{for all}\ t>0,\ \ \text{as}\ n\longrightarrow +\infty.
\end{equation}
 Passing to the limit in \eqref{4eq9} as $n\longrightarrow +\infty$, and using \eqref{4eq85} and \eqref{8eq1}, we get
$$ \frac{1}{2}t^2\leq c_0<+\infty,\ \text{for all}\ t\geq0.$$
Therefore, $v$ could not be zero ($v\not\equiv 0$).

Using \eqref{4eq300} and Lemma \ref{lem1}-(10), we obtain
\begin{align*}
J( u_n) & =J(\Vert u_n\Vert v_n )\\
 & =\frac{1}{2} \left\Vert \Vert u_n\Vert v_n\right\Vert^2+\frac{q^2}{4}\int_{\mathbb{R}^3}\phi_{\Vert u_n\Vert v_n}(x) (\Vert u_n\Vert v_n)^2dx-\int_{\mathbb{R}^3}F(x,\Vert u_n\Vert v_n)dx\nonumber\\
 & = \frac{1}{2}  \Vert u_n\Vert ^2+\frac{q^2}{4}\Vert u_n\Vert ^4\int_{\mathbb{R}^3}\phi_{v_n} v_n^2dx-\int_{\mathbb{R}^3}F(x,\Vert u_n\Vert v_n)dx.\nonumber\\
 & \leq \frac{1}{2}  \Vert u_n\Vert ^2+\frac{q^2C}{4a}\Vert u_n\Vert ^4-\int_{\mathbb{R}^3}F(x,\Vert u_n\Vert v_n)dx.
\end{align*}
It follows that
\begin{align}\label{eq6}
\frac{J( u_n)}{\Vert u_n\Vert^{4} } & \leq \frac{1}{2\Vert u_n\Vert^{2}} +\frac{q^2C}{4a}-\int_{\mathbb{R}^{3}}\frac{F(x,\Vert u_n\Vert v_n)}{\Vert u_n\Vert^{4}} dx.
\end{align}
Exploiting assumption $(f_3)$, Fatou's lemma and the fact that $v\not\equiv 0$, we infer that
\begin{align}\label{eq7}
\liminf_{n\rightarrow +\infty}\int_{\mathbb{R}^{3}}\frac{F(x,\Vert u_n\Vert v_n)}{\Vert u_n\Vert^{4}} dx & =\liminf_{n\rightarrow +\infty}\int_{\mathbb{R}^{3}}\frac{F(x,\Vert u_n\Vert v_n)}{\left( \Vert u_n\Vert \vert v_n\vert\right) ^{4}}\vert v_n\vert^{4} dx =+\infty.
\end{align}
Thus, from \eqref{eq6}, it yields that
$$\frac{J( u_n)}{\Vert u_n\Vert^{4}}\longrightarrow - \infty\ \text{as}\ \ n\longrightarrow +\infty,$$
which is a contradiction with \eqref{4eq85}. Therefore, $\lbrace u_n\rbrace_{n\in\mathbb{N}}$ is bounded in $H^{1}_V(\mathbb{R}^3)$. This completes the proof of Claim.\\
Hence, up to a subsequence, there exists $\widehat{u}\in H^{1}_V(\mathbb{R}^3)$  such that
\begin{equation}\label{4eq555}
u_n\rightharpoonup\widehat{u}\ \text{in}\ H^{1}_V(\mathbb{R}^3),\ \ \text{and}\ \ u_n(x)\longrightarrow \widehat{u}(x)\ \text{ as}\ n\longrightarrow +\infty,\ \text{for a.a.}\ x\in\mathbb{R}^3.
\end{equation}
It follows, from Lemma \ref{lem5+}, that $$\widehat{u} \not\equiv 0.$$ Thus, according to Lemma \ref{lem3}, there exists a unique $t_{\widehat{u}}>0$ such that
\begin{equation}\label{4eq122}
t_{\widehat{u}}\widehat{u}\in \mathcal{N}.
\end{equation}
By \eqref{4eq555}, Lemmas \ref{lem1}, \ref{9lem1}, \ref{6lem4}, and Fatou's lemma,  it follows that
\begin{align}\label{4eq108}
c_0=\lim_{n\rightarrow +\infty}J(u_n)& \geq \liminf_{n\rightarrow +\infty} J(t_{\widehat{u}}u_n)\geq J(t_{\widehat{u}}\widehat{u})\nonumber\\
& \geq c_0.
\end{align}
Therefore, $c_0=J(t_{\widehat{u}}\widehat{u})=\inf\limits_{\mathcal{N}}J.$\\
Now, we show that $t_{\widehat{u}}=1$. Indeed, since $\lbrace u_n\rbrace_{n\in\mathbb{N}}\in\mathcal{N}$,
\begin{align*}
\int_{\mathbb{R}^3}\vert\nabla u_n\vert^2 dx +\int_{\mathbb{R}^3} V(x)u_n^2dx+q^2\int_{\mathbb{R}^3} \phi_{u_n} u_n^2dx=\int_{\mathbb{R}^3}f(x,u_n)u_ndx,\ \text{for all}\ n\in\mathbb{N}.
\end{align*}
By \eqref{4eq555}, Lemma \ref{lem1}-(8),  Fatou's lemma, and Lemma \ref{9lem1}, we deduce that
\begin{align}\label{4eq120}
\int_{\mathbb{R}^3}\vert\nabla \widehat{u}\vert^2 dx +\int_{\mathbb{R}^3} V(x)\widehat{u}^2dx+q^2\int_{\mathbb{R}^3} \phi_{\widehat{u}} \widehat{u}^2dx\leq \int_{\mathbb{R}^3}f(x,\widehat{u})\widehat{u}dx.
\end{align}
Suppose that $t_{\widehat{u}}>1$.
From \eqref{4eq122} and Lemma \ref{lem1}, one has
\begin{align}\label{4eq800}
\displaystyle{\int_{\mathbb{R}^{3}}\frac{f(x,t_{\widehat{u}}\widehat{u})\vert \widehat{u}\vert^{3}\widehat{u}}{\vert t_{\widehat{u}}\widehat{u}\vert^{3}}\ dx} & = \frac{1}{t_{\widehat{u}}^2}\left(\int_{\mathbb{R}^3}\vert\nabla \widehat{u}\vert^2 dx +\int_{\mathbb{R}^3} V(x)\widehat{u}^2dx\right)+q^2\int_{\mathbb{R}^3} \phi_{\widehat{u}} \widehat{u}^2dx.
\end{align}
Subtracting \eqref{4eq800} from \eqref{4eq120}, and using  \eqref{1}, we find that
\begin{align}\label{h3}
 0<\int_{\mathbb{R}^{3}}\left[ \frac{f(x,\widehat{u})}{\vert \widehat{u}\vert^{3}}- \frac{f(x,t_{\widehat{u}}\widehat{u})}{\vert t_{\widehat{u}} \widehat{u}\vert^{3}}\right]\vert \widehat{u}\vert^{3}\widehat{u} \ dx\leq 0,
\end{align}
which is a contradiction. Therefore, $0<t_u\leq 1$.\\
Suppose that $t_{\widehat{u}}\neq 1$. Using \eqref{2}, Lemma \ref{9lem1} and Fatou's lemma, we see that
\begin{align*}
c_0&=J(t_{\widehat{u}}\widehat{u})=J(t_{\widehat{u}}\widehat{u})-\frac{1}{4}\langle J^{'}(t_{\widehat{u}}\widehat{u}),t_{\widehat{u}}\widehat{u}\rangle\\
& =\frac{1}{2} \Vert t_{\widehat{u}}\widehat{u}\Vert^2+\frac{q^2}{4}\int_{\mathbb{R}^3}\phi_{t_{\widehat{u}}\widehat{u}}(x) (t_{\widehat{u}}\widehat{u})^2dx-\int_{\mathbb{R}^3}F(x,t_{\widehat{u}}\widehat{u})dx-\frac{1}{4} \Vert t_{\widehat{u}}\widehat{u}\Vert^2\\
&-\frac{q^2}{4}\int_{\mathbb{R}^3}\phi_{t_{\widehat{u}}\widehat{u}}(x) (t_{\widehat{u}}\widehat{u})^2dx+\frac{1}{4}\int_{\mathbb{R}^3} f(x,t_{\widehat{u}}\widehat{u})t_{\widehat{u}}\widehat{u}dx\\
&=\frac{1}{4} \Vert t_{\widehat{u}}\widehat{u}\Vert^2+\int_{\mathbb{R}^3} \left[ \frac{1}{4}f(x,t_{\widehat{u}}\widehat{u})t_{\widehat{u}}\widehat{u}-F(x,t_{\widehat{u}}\widehat{u}) \right]dx\\
& < \frac{1}{4} \Vert \widehat{u}\Vert^2+\int_{\mathbb{R}^3} \left[ \frac{1}{4}f(x,\widehat{u})\widehat{u}-F(x,\widehat{u}) \right]dx\\
&\leq\liminf_{n\rightarrow +\infty}\left( \frac{1}{4} \Vert u_n\Vert^2+\int_{\mathbb{R}^3} \left[ \frac{1}{4}f(x,u_n)u_n-F(x,u_n) \right]dx\right) \\
& =\liminf_{n\rightarrow +\infty}\left(J(u_n)-\frac{1}{4}\langle J^{'}(u_n),u_n\rangle\right)
=\lim\limits_{n\rightarrow +\infty} J(u_n)=J(\widehat{u})\\
&=c_0.
\end{align*}
Thus, a contradiction holds. Consequently, $t_{\widehat{u}}=1$. Hence,
$$c_0=J(\widehat{u})=\inf\limits_{\mathcal{N}}J.$$
This completes the proof.
\end{proof}
The next proposition shows that the Nehari manifold $\mathcal{N}$ is a natural constraint for $\widehat{u}$.
\begin{prop}\label{4prop9}
Assume that assumptions  $(f_1)-(f_4)$ and $(V_0)$ hold. Then, $\widehat{u}$ is a critical point of $J$.
\end{prop}
\begin{proof}
First, we consider the functional $\xi:H^{1}_V(\mathbb{R}^3)\longrightarrow \mathbb{R}$ defined by
\begin{align*}
\xi(u) & :=\langle J^{'}(u),u\rangle =\Vert u\Vert^2+q^2\int_{\mathbb{R}^3} \phi_u u^2dx-\int_{\mathbb{R}^3}f(x,u)udx.
\end{align*}
According to assumptions  $(f_1)-(f_4)
$, $\xi$  is locally Lipschitz (see \cite[Theorem 2.7.2, p. 221]{419}).  Then, the generalized subdifferential of $\xi(\cdot)$ at all $u\in H^{1}_V(\mathbb{R}^3)$ is the set $\partial\xi(u)\subseteq (H^{1}_V(\mathbb{R}^3))^*$ such that for all $\xi_{u}^*\in \partial\xi(u)$ there exists $f^*(x,\cdot)\in \partial f(x,\cdot)$ verifies
\begin{align}\label{4eq500}
\langle\xi^*_u,v\rangle
& =2\int_{\mathbb{R}^3}\nabla u.\nabla v dx+2\int_{\mathbb{R}^3} V(x)uvdx- \int_{\mathbb{R}^{3}} f(x,u)v dx- \int_{\mathbb{R}^{3}} f^*(x,u)uv dx\nonumber\\
& + 2q^2\int_{\mathbb{R}^3}\int_{\mathbb{R}^3}\frac{1-e^{-\frac{\vert y-x\vert}{a}}
   }{\vert y-x\vert}u(y)v(y)u(x)^2dydx +2q^2\int_{\mathbb{R}^3} \phi_u uvdx,
\end{align} for all $v\in H^{1}_V(\mathbb{R}^3)$.\\
From the definition of $\xi(\cdot)$ and Proposition \ref{4prop5}, we have
$$J(\widehat{u})=c_0=\inf\left\lbrace J(u):\ \xi(u)=0,\ u\in H^{1}_V(\mathbb{R}^3)\backslash\lbrace 0\rbrace \right\rbrace.$$
Using the non-smooth multiplier rule of Clarke \cite[Theorem 10.47, p. 221]{418}, we find $\lambda\geq 0$  such that
$$
0\in\partial(J+ \lambda \xi)(\widehat{u}).
$$
 By the subdifferential calculus of Clarke \cite[p. 48]{419}, it follows that

$$0\in\partial J(\widehat{u})+\lambda\partial\xi(\widehat{u}).$$
Thus,
\begin{equation}\label{4eq11}
0=J^{'}(\widehat{u})+\lambda \xi^*_{\widehat{u}}\ \ \text{in}\ (H^{1}_V(\mathbb{R}^3))^*,\ \ \text{for all}\ \xi^*_{\widehat{u}}\in \partial\xi(\widehat{u}).
\end{equation}
 Since $\widehat{u}\in \mathcal{N}$, we have
 \begin{equation}\label{4eq10}
0=\langle J^{'}(\widehat{u}),\widehat{u}\rangle+\lambda\langle \xi^{*}_{\widehat{u}},\widehat{u}\rangle =\lambda\langle \xi^{*}_{\widehat{u}},\widehat{u}\rangle,\ \ \text{for all}\ \xi^*_{\widehat{u}}\in \partial\xi(\widehat{u}).
\end{equation}
Using \eqref{4eq500}, assumption $(f_4)$, and the fact that $\widehat{u}\in\mathcal{N}$, for all $f^*(x,\cdot)\in \partial f(x,\cdot)$, we get
\begin{align}\label{4eq12}
 \langle \xi^{*}_{\widehat{u}},\widehat{u}\rangle
 & = 2\int_{\mathbb{R}^3}\vert \nabla \widehat{u}\vert^2dx+2\int_{\mathbb{R}^3} V(x)\widehat{u}^2dx- \int_{\mathbb{R}^{3}} f(x,\widehat{u})\widehat{u} dx- \int_{\mathbb{R}^{3}} f^*(x,\widehat{u})\widehat{u}^{2} dx\\
& + 2q^2\int_{\mathbb{R}^3}\int_{\mathbb{R}^3}\frac{1-e^{-\frac{\vert y-x\vert}{a}}
   }{\vert y-x\vert}\widehat{u}(y)^2\widehat{u}(x)^2dydx +2q^2\int_{\mathbb{R}^3} \phi_{\widehat{u}} \widehat{u}^2dx\nonumber\\
 & = 2\int_{\mathbb{R}^3}\vert \nabla \widehat{u}\vert^2dx+2\int_{\mathbb{R}^3} V(x)\widehat{u}^2dx- \int_{\mathbb{R}^{3}} f(x,\widehat{u})\widehat{u} dx- \int_{\mathbb{R}^{3}} f^*(x,\widehat{u})\widehat{u}^{2} dx\\
& + 4q^2\int_{\mathbb{R}^3} \phi_{\widehat{u}} \widehat{u}^2dx\nonumber\\
& = -2\int_{\mathbb{R}^3}\vert \nabla\widehat{u}\vert^2dx-2\int_{\mathbb{R}^3} V(x)\widehat{u}^2dx+3 \int_{\mathbb{R}^{3}} f(x,\widehat{u})\widehat{u} dx- \int_{\mathbb{R}^{3}} f^*(x,\widehat{u})\widehat{u}^{2} dx\nonumber\\
& = -2\Vert \widehat{u}\Vert+ \int_{\mathbb{R}^{3}} 3f(x,\widehat{u})\widehat{u}-f^*(x,\widehat{u})\widehat{u}^{2} dx\nonumber\\
& <0.
\end{align}
It follows, from \eqref{4eq10}, that $\lambda=0$. Therefore, by \eqref{4eq11}, we deduce that
$$J^{'}(\widehat{u})=0\ \text{in}\ (H^{1}_V(\mathbb{R}^3))^*.$$
Hence, $\widehat{u}$ is a critical point of $J$, so, it is a weak solution of problem \eqref{P}.\\
Thus the proof.
\end{proof}
\begin{proof}[\textbf{Proof of Theorem \ref{thm1} :}]  Theorem \ref{thm1} deduced from the Propositions \ref{4prop5} and \ref{4prop9}.
\end{proof}
\section{Least energy nodal solution}
In this section, we establish the existence of a least energy nodal solution for problem \eqref{P} and  we give the proof of Theorem \ref{thm2}.\\

In order to find a least energy nodal solution for problem \eqref{P}, we look for a minimizer of the energy functional $J$ on the constraint  $\mathcal{M}$ ( where $\mathcal{M}$ was defined in Sect. 2 ). Let's consider the following minimization problem
\begin{equation}\label{M1}
c_1:=\inf_{\mathcal{M}}J.
\end{equation}
\begin{rem}\label{rem4}
Since $\mathcal{M}\subset\mathcal{N}$, by Lemma \ref{lem6}, we have
$$c_1=\inf_{\mathcal{M}}J\geq \inf_{\mathcal{N}}J=c_{0}>0.$$
\end{rem}
\begin{lemma}\label{4lem11}
Assume that assumptions $(f_1)-(f_4)$ and $(V_0)$ hold. Then, for each $w\in H^1_V(\mathbb{R}^{3})$ such that $w^{\pm}\not\equiv 0$, there exists a unique pair $(t_{w^+},s_{w^-})\in (0,+\infty)\times(0,+\infty)$ such that
$$t_{w^+}w^++s_{w^-}w^-\in\mathcal{M}.$$
\end{lemma}
\begin{proof}
Let $\xi:\ (0,+\infty)\times  (0,+\infty) \longrightarrow \mathbb{R}^{2}$ be a continuous vector field given by
$$\xi(t,s)=\big{(}\xi_{1}(t,s),\xi_{2}(t,s)\big{)},\ \ \text{for all}\ \ t,s \in (0,+\infty)\times (0,+\infty)$$
where $$\xi_{1}(t,s)=\langle J'(tw^{+}+sw^{-}),tw^{+}\rangle\ \text{ and}\ \xi_{2}(t,s)=\langle J'(tw^{+}+sw^{-}),sw^{-}\rangle.$$
From \eqref{1}, for $t\in (0,1)$, we have
$$\frac{f(x,tu)}{\vert tu\vert^3}\leq \frac{f(x,u)}{\vert u\vert^3},\ \text{for a.a.}\ x\in \mathbb{R}^3,\ \text{with}\ \vert u(x)\vert>0.$$
Hence,
\begin{equation}\label{8eq77}
    f(x,tu(x))u(x)\leq t^3f(x,u(x))u(x),\ \text{for a.a.}\ x\in \mathbb{R}^3.
\end{equation}
It follows, by the definitions of $\xi_1(\cdot,\cdot)$ and $\xi_2(\cdot,\cdot)$, for all $t\in (0,1)$, that
\begin{align*}
  \xi_1(t,t)&= \langle J'(tw),tw^{+}\rangle \\
  & =t^2\left[\int_{\mathbb{R}^3}\vert\nabla w^+\vert^2 dx +\int_{\mathbb{R}^3} V(x)(w^+)^2dx\right]+q^2t^4\int_{\mathbb{R}^3} \phi_w (x)(w^+)^2dx-\int_{\mathbb{R}^3}f(x,tw^+)tw^+dx\\
  & \geq t^2\left[\int_{\mathbb{R}^3}\vert\nabla w^+\vert^2 dx +\int_{\mathbb{R}^3} V(x)(w^+)^2dx\right]+q^2t^4\int_{\mathbb{R}^3} \phi_w (x)(w^+)^2dx-t^4\int_{\mathbb{R}^3}f(x,w^+)w^+dx,
\end{align*}
and
\begin{align*}
  \xi_2(t,t)&= \langle J'(tw),tw^{-}\rangle \\
  & =t^2\left[\int_{\mathbb{R}^3}\vert\nabla w^-\vert^2 dx +\int_{\mathbb{R}^3} V(x)(w^-)^2dx\right]+q^2t^4\int_{\mathbb{R}^3} \phi_w (x)(w^-)^2dx-\int_{\mathbb{R}^3}f(x,tw^-)tw^-dx\\
  & \geq t^2\left[\int_{\mathbb{R}^3}\vert\nabla w^-\vert^2 dx +\int_{\mathbb{R}^3} V(x)(w^-)^2dx\right]+q^2t^4\int_{\mathbb{R}^3} \phi_w (x)(w^-)^2dx-t^4\int_{\mathbb{R}^3}f(x,w^-)w^-dx.
\end{align*}
Thus, there exists $r_1>0$ small enough  such that
\begin{equation}\label{4eq90}
\xi_1(t,t)>0,\ \ \xi_2(t,t)>0,\ \ \text{for all}\ t\in(0,r_1).
\end{equation}
On the other side, for $t>0$, we have
\begin{align}\label{9eq55}
   \frac{\xi_1(t,t)}{t^4}=\frac{1}{t^2}\left[\int_{\mathbb{R}^3}\vert\nabla w^+\vert^2 dx +\int_{\mathbb{R}^3} V(x)(w^+)^2dx\right]+q^2\int_{\mathbb{R}^3} \phi_w(x) (w^+)^2dx-\int_{\mathbb{R}^3}\frac{f(x,tw^+)}{t^4}tw^+dx
\end{align}
and
\begin{align}\label{9eq555}
   \frac{\xi_2(t,t)}{t^4}=\frac{1}{t^2}\left[\int_{\mathbb{R}^3}\vert\nabla w^-\vert^2 dx +\int_{\mathbb{R}^3} V(x)(w^-)^2dx\right]+q^2\int_{\mathbb{R}^3} \phi_w (x)(w^-)^2dx-\int_{\mathbb{R}^3}\frac{f(x,tw^-)}{t^4}tw^-dx.
\end{align}
From \eqref{2} and assumption $(f_3)$, we have
\textcolor{red}{$$\int_{\mathbb{R}^3}\frac{f(x,tw^\pm)}{t^3}w^\pm dx=\int_{\mathbb{R}^3}\frac{f(x,tw^\pm)}{t^4}tw^\pm dx\longrightarrow +\infty\ \text{as}\ t\longrightarrow +\infty.$$}
It follows, by passing to the limit as $t\rightarrow +\infty $ in \eqref{9eq55} and \eqref{9eq555}, that
$$\lim\limits_{t\rightarrow +\infty}\frac{\xi_1(t,t)}{t^4}=\lim\limits_{t\rightarrow +\infty}\frac{\xi_2(t,t)}{t^4}=-\infty.$$
Thus, there exists $R_1>0$ large enough such that
\begin{equation}\label{4eq901}
\xi_1(t,t)<0,\ \ \xi_2(t,t)<0,\ \ \text{for all}\ t\in(R_1,+\infty).
\end{equation}
\textbf{Claim:}
 $\xi_1(t,s)$ is nondecreasing in $s$ on $(0,+\infty)$ for fixed $t>0$ and $\xi_2(t,s)$ is nondecreasing in $t$ on $(0,+\infty)$ for fixed $s>0$. Indeed, Let $t,s_1,s_2>0$  such that $s_1\leq s_2$. By the definition of $\xi_1(\cdot,\cdot)$ and Lemma \ref{lem1}, we have
 \begin{align*}
 \xi_1(t,s_2)-\xi_1(t,s_1)&=\langle J'(tw^{+}+s_2w^{-}),tw^{+}\rangle-\langle J'(tw^{+}+s_1w^{-}),tw^{+}\rangle\\
&=q^2\left[\int_{\mathbb{R}}\phi_{tw^{+}+s_2w^{-}}(x)(tw^+)^2dx-\int_{\mathbb{R}}\phi_{tw^{+}+s_1w^{-}}(x)(tw^+)^2dx\right]\\
&=q^2\left[\int_{\mathbb{R}}\phi_{s_2w^{-}}(x)(tw^+)^2dx-\int_{\mathbb{R}}\phi_{s_1w^{-}}(x)(tw^+)^2dx\right]\\
&=q^2\left[s_2^2t^2\int_{\mathbb{R}}\phi_{w^{-}}(x)(w^+)^2dx-s_1^2t^2\int_{\mathbb{R}}\phi_{w^{-}}(x)(w^+)^2dx\right]\\
&\geq 0.
  \end{align*}
  Therefore, the map $s\mapsto \xi_1(t,s)$ is nondecreasing on $(0,+\infty)$, for a fixed $t>0$. Similarly, we prove that the map $t\mapsto \xi_2(t,s)$ is nondecreasing on $(0,+\infty)$, for a fixed $s>0$.\\
  This ends the proof of Claim.\\

Now, exploiting \eqref{4eq90}, \eqref{4eq901} and the \textbf{Claim}, we could exist $r>0$ and $R>0$ such that $r<R$ and
$$\xi_1(r,s)>0,\ \ \xi_1(R,s)<0,\ \ \text{for all}\ s\in(r,R]$$
and
$$\xi_2(t,r)>0,\ \ \xi_2(t,R)<0,\ \ \text{for all}\ t\in(r,R].$$
It follows, by applying  Miranda's theorem \cite{412} on $\xi$, that there exist some $t_{w^+},s_{w^-} \in (r,R]$ such that $\xi(t_{w^+},s_{w^-})=\left(\xi_1(t_{w^+},s_{w^-}),\xi_2(t_{w^+},s_{w^-})\right)=(0,0)$. Hence, $$t_{w^+}w^++s_{w^-}w^-\in\mathcal{M}.$$

For the uniqueness of the pairs $(t_{w^+},s_{w^-})$, we argue by contradiction. Suppose that there exist two different pairs $(t_1,s_1)$ and $(t_2,s_2)$ such that
$$t_1w^++s_1w^-\in \mathcal{M}\ \text{and}\ t_2w^++s_2w^-\in \mathcal{M}.$$
We distinguish two cases:\\
\textbf{Case 1:} $w\in \mathcal{M}$. Without loss of generality, we may take $(t_1,s_1)=(1,1)$ and assume that $t_2\leq s_2$, we have
\begin{align}\label{4eq210}
\int_{\mathbb{R}^3}\vert\nabla w^+\vert^2 dx +\int_{\mathbb{R}^3} V(x)(w^+)^2dx+q^2\int_{\mathbb{R}^3} \phi_w (x)(w^+)^2dx=\int_{\mathbb{R}^3}f(x,w^+)w^+dx,
\end{align}
\begin{align}\label{4eq2101}
\int_{\mathbb{R}^3}\vert\nabla w^-\vert^2 dx +\int_{\mathbb{R}^3} V(x)(w^-)^2dx+q^2\int_{\mathbb{R}^3} \phi_w (x)(w^-)^2dx=\int_{\mathbb{R}^3}f(x,w^-)w^-dx,
\end{align}
\begin{align}\label{4eq210+}
\int_{\mathbb{R}^3}\vert\nabla t_2w^+\vert^2 dx +\int_{\mathbb{R}^3} V(x)(t_2w^+)^2dx+q^2\int_{\mathbb{R}^3} \phi_{t_2w^++s_2w^-}(x) (t_2w^+)^2dx=\int_{\mathbb{R}^3}f(x,t_2w^+)t_2w^+dx
\end{align}
and
\begin{align}\label{4eq210-}
\int_{\mathbb{R}^3}\vert\nabla s_2w^-\vert^2 dx +\int_{\mathbb{R}^3} V(x)(s_2w^-)^2dx+q^2\int_{\mathbb{R}^3} \phi_{t_2w^++s_2w^-}(x)(s_2w^-)^2dx=\int_{\mathbb{R}^3}f(x,s_2w^-)s_2w^-dx.
\end{align}
On the other hand, from the definition of $\phi_u$, Lemma \ref{lem1} and the fact that $\text{Supp}(w^+)\cap\text{Supp}(w^+)=\emptyset$, we see that
\begin{align}\label{8eq88}
    \int_{\mathbb{R}^3} \phi_{t_2w^++s_2w^-}(x) (t_2w^+)^2dx&=\int_{\mathbb{R}^3} \phi_{t_2w^+}(x) (t_2w^+)^2dx+\int_{\mathbb{R}^3} \phi_{s_2w^-}(x) (t_2w^+)^2dx\nonumber\\
    &=t_2^4\int_{\mathbb{R}^3} \phi_{w^+}(x) (w^+)^2dx+t_2^2s_2^2\int_{\mathbb{R}^3} \phi_{w^-}(x) (w^+)^2dx\nonumber\\
    & \geq t_2^4 \int_{\mathbb{R}^3} \phi_{w}(x) (w^+)^2dx\ (\text{since}\ t_2\leq s_2)
\end{align}
and
\begin{align}\label{8eq888}
    \int_{\mathbb{R}^3} \phi_{t_2w^++s_2w^-}(x) (s_2w^-)^2dx&=\int_{\mathbb{R}^3} \phi_{t_2w^+}(x) (s_2w^-)^2dx+\int_{\mathbb{R}^3} \phi_{s_2w^-}(x) (s_2w^-)^2dx\nonumber\\
    &=t_2^2s_2^2\int_{\mathbb{R}^3} \phi_{w^+}(x) (w^-)^2dx+s_2^4\int_{\mathbb{R}^3} \phi_{w^-}(x) (w^-)^2dx.\nonumber\\
    &\leq s_2^4\int_{\mathbb{R}^3} \phi_{w}(x) (w^-)^2dx\ (\text{since}\ t_2\leq s_2).
\end{align}
Putting together \eqref{4eq210+} with \eqref{8eq88} and \eqref{4eq210-} with \eqref{8eq888}, we obtain respectively
\begin{align}\label{4eq210++}
\frac{1}{t_2^2}\left[\int_{\mathbb{R}^3}\vert\nabla w^+\vert^2 dx +\int_{\mathbb{R}^3} V(x)(w^+)^2dx\right]+q^2\int_{\mathbb{R}^3} \phi_{w}(x) (w^+)^2dx\leq \int_{\mathbb{R}^3}\frac{f(x,t_2w^+)}{\vert t_2 w^+\vert^3}\vert w^+\vert^3w^+dx
\end{align}
and
\begin{align}\label{4eq210--}
\frac{1}{s_2^2}\left[\int_{\mathbb{R}^3}\vert\nabla w^-\vert^2 dx +\int_{\mathbb{R}^3} V(x)(w^-)^2dx\right]+q^2\int_{\mathbb{R}^3} \phi_{w}(x) (w^-)^2dx\geq \int_{\mathbb{R}^3}\frac{f(x,s_2w^-)}{\vert s_2 w^-\vert^3}\vert w^-\vert^3w^-dx.
\end{align}
Now, we shall prove that the following six cases could not happen
\begin{equation*}
  \begin{array}{|ll|}
  \hline
  \   & \ \\
 (1) & t_2=s_2<1.\\
\   & \ \\
\hline
\   & \ \\
(2) & 1<t_2=s_2.\\
\   & \ \\
\hline
\   & \ \\
(3) & 0<t_2<s_2\leq 1.\\
\ & \ \\
\hline
  \end{array}
  \begin{array}{|ll|}
  \hline
  \   & \ \\
 (4) & 1\leq t_2<s_2.\\
\   & \ \\
\hline
\   & \ \\
(5) & 0< t_2\leq 1 <s_2.\\
\   & \ \\
\hline
\   & \ \\
(6) & 0< t_2< 1 \leq s_2.\\
\ & \ \\
\hline
  \end{array}
\end{equation*}
 Suppose that one of the cases $(1)$ or $(3)$ or $(6)$, holds. It follows, by subtracting  \eqref{4eq210} from \eqref{4eq210++}, and using \eqref{1}, that
\begin{align*}
0&< \left(\frac{1}{t_2^2}-1\right)\left[\int_{\mathbb{R}^3}\vert\nabla w^+\vert^2 dx +\int_{\mathbb{R}^3} V(x)(w^+)^2dx\right]\\
& \leq \int_{\mathbb{R}^3}\left[\frac{f(x,t_2w^+)}{\vert t_2w^+\vert^{3}}-\frac{f(x,w^+)}{\vert w^+\vert^{3}} \right]\vert w^+\vert^{3}w^+dx\\
&\leq 0.
\end{align*}
Thus, a contradiction holds. Then, the cases $(1)$, $(3)$ and $(6)$ cannot be realized.\\
Suppose that one of the cases $(2)$ or $(4)$ or $(5)$, holds. It follows, by subtracting  \eqref{4eq2101} from \eqref{4eq210--}, and using \eqref{1}, that
\begin{align*}
0&\leq \int_{\mathbb{R}^3}\left[\frac{f(x,s_2w^-)}{\vert s_2w^-\vert^{3}}-\frac{f(x,w^-)}{\vert w^-\vert^{3}} \right]\vert w^-\vert^{3}w^-dx\\
& \leq \left(\frac{1}{s_2^2}-1\right)\left[\int_{\mathbb{R}^3}\vert\nabla w^-\vert^2 dx +\int_{\mathbb{R}^3} V(x)(w^-)^2dx\right]\\
&<0.
\end{align*}
Which gives also a contradiction. Then, the cases $(2)$, $(4)$ and $(5)$ cannot be happen.
We recap that $(t_1,s_1)=(1,1)=(t_2,s_2)$. \\
\textbf{Case 2:} $w \notin \mathcal{M}$. Let $v=t_1w^++s_1w^-\in \mathcal{M}$, $v^+=t_1w^+$ and $v^-=s_1w^-$, so $(t_1,s_1)\neq (1,1)$. It is clear that
$$t_2w^++s_2w^-=\frac{t_2}{t_1}t_1w^++\frac{s_2}{s_1}s_1w^-=\frac{t_2}{t_1}v^++\frac{s_2}{s_1}v^-\in \mathcal{M}.$$
Proceeding as in \textbf{Case 1}, we conclude that
$$\frac{t_2}{t_1}=\frac{s_2}{s_1}=1.$$
This completes the proof.
\end{proof}

\begin{lemma}\label{4lem16}
Assume that assumptions $(f_1)-(f_4)$ and  $(V_0)$ hold. Then, for all $w\in\mathcal{M}$,
$$J(tw^++sw^-)\leq J(w),\  \text{for all}\ t,s> 0.$$
\end{lemma}
\begin{proof}
For each $w\in\mathcal{M}$, we consider the fibering map $\mu_w:(0,+\infty)\times (0,+\infty)\longrightarrow \mathbb{R}$ defined by
$$\mu_w(t,s)=J(tw^++sw^-)\ \ \text{for all}\ t,s> 0.$$
In light of Lemma \ref{4lem11} and Remark \ref{rem4},
\begin{equation}\label{4eq994}
\mu_w(0,0)=J(0)=0<c_1\leq \mu_w(1,1)=J(w).
\end{equation}
Let $t,s>0$. Using Lemmas \ref{lem1} and Sobolev embedding theorem, we obtain
\begin{align}\label{4eq18}
\mu_w(t,s)&=\frac{1}{2} \Vert tw^++sw^-\Vert^2+\frac{q^2}{4}\int_{\mathbb{R}^3}\phi_{tw^++sw^-} (tw^++sw^-)^2dx-\int_{\mathbb{R}^3}F(x,tw^++sw^-)dx\nonumber\\
&\leq \frac{1}{2} \Vert tw^++sw^-\Vert^2+\frac{q^2}{4a}\Vert tw^++sw^-\Vert^4_2-\int_{\mathbb{R}^3}F(x,tw^++sw^-)dx\nonumber\\
&\leq \frac{1}{2} (t+s)^2\left(\Vert w^+\Vert+\Vert w^-\Vert\right)^2+\frac{q^2C}{4a}\Vert tw^++sw^-\Vert^4-\int_{\mathbb{R}^3}F(x,tw^++sw^-)dx\nonumber\\
&\leq \frac{1}{2} (t+s)^2\left(\Vert w^+\Vert+\Vert w^-\Vert\right)^2+\frac{q^2C}{4a}(t+s)^4\left(\Vert w^+\Vert+\Vert w^-\Vert\right)^4-\int_{\mathbb{R}^3}F(x,tw^++sw^-)dx.
\end{align}
It follows that
\begin{align}\label{4eq19}
\frac{\mu_w(t,s)}{(t+s)^4}&\leq \frac{1}{(t+s)^2}\left(\Vert w^+\Vert+\Vert w^-\Vert\right)^2+\frac{q^2C}{4a}\left(\Vert w^+\Vert+\Vert w^-\Vert\right)^4\nonumber\\ &-\int_{\mathbb{R}^d} \frac{F(x,tw^++sw^-)}{(t+s)^4}dx.
\end{align}
By assumption $(f_3)$ and the fact that $\text{supp}(w^+)\cap\text{supp}(w^-)=\emptyset$, we infer that
\begin{equation}\label{4eq20}
\lim\limits_{\vert (t,s)\vert\rightarrow +\infty}\frac{F(x,tw^++sw^-)}{(t+s)^4}=+\infty,\ \ \text{for a.a.}\ x\in\mathbb{R}^d.
\end{equation}
Using \eqref{4eq19} and \eqref{4eq20}, we deduce that
 $$\limsup\limits_{\vert (t,s)\vert\rightarrow +\infty} \mu_w(t,s)= -\infty.$$
Thus, by \eqref{4eq994}, the map $\mu_w(\cdot,\cdot)$ has  a global maximum $(t_{w^+},s_{w^-})\in (0,+\infty)\times (0,+\infty)$. Hence, $(t_{w^+},s_{w^-})$ is a critical point for $\mu_w(\cdot,\cdot)$, that is,

$$\langle J^{'}(t_{w^+}w^++s_{w^-}w^-),w^+\rangle=0=\langle J^{'}(t_{w^+}w^++s_{w^-}w^-),w^-\rangle.$$
By  Lemma \ref{4lem16} and the fact that $w\in \mathcal{M}$,
$$(t_{w^+},s_{w^-})=(1,1).$$
Therefore,
$$J(tw^++sw^-)\leq J(t_{w^+}w^++s_{w^-}w^-)= J(w),\ \text{for all}\ t,s>0.$$
This ends the proof.
\end{proof}
\begin{lemma}\label{lem5}
Assume that assumptions $(f_1)-(f_4)$  and $(V_0)$ hold. Let $\lbrace u_n\rbrace_{n\in\mathbb{N}}\subset \mathcal{M}$ such that $u_n\rightharpoonup u$ in $H^{1}_V(\mathbb{R}^3)$, then $u^\pm\not\equiv 0.$
\end{lemma}
\begin{proof}
\textbf{Claim}: There is $\tau>0$ such that
\begin{equation}\label{7eq88+1}
    \Vert v^\pm\Vert\geq \tau,\ \ \text{for all}\ v\in \mathcal{M}.
\end{equation}
Indeed: Let $v\in\mathcal{M}$. Then,
\begin{align*}
  \int_{\mathbb{R}^3}\vert\nabla v^{\pm}\vert^2 dx +\int_{\mathbb{R}^3} V(x)(v^{\pm})^2dx+q^2\int_{\mathbb{R}^3} \phi_v (v^{\pm})^2dx=\int_{\mathbb{R}^3}f(x,v^{\pm})v^{\pm}dx.
\end{align*}
It follows, by Lemma \ref{lem1}, that
\begin{align}\label{8eq99}
    \Vert v^{\pm}\Vert^2\leq \int_{\mathbb{R}^3}f(x,v^{\pm})v^{\pm}dx.
\end{align}
Using Theorem \ref{thm3}, \eqref{8eq999} and \eqref{8eq99}, we find that
\begin{align*}
   \Vert v^{\pm}\Vert^2 &\leq \varepsilon\int_{\mathbb{R}^3} \vert v^\pm\vert^2dx+C_\varepsilon\int_{\mathbb{R}^3} \vert v^\pm\vert^6dx\\
    & \leq \varepsilon C_1 \Vert v^{\pm}\Vert^2+ C_\varepsilon C_2 \Vert v^{\pm}\Vert^6,\ \text{for some constants}\ C_1,C_2>0.
\end{align*}
 Choosing $\varepsilon=\frac{1}{2C_1}$ in the previous inequality, we get
$$\frac{1}{2}\Vert v^\pm\Vert^2\leq C_\varepsilon C_2 \Vert v^{\pm}\Vert^6.$$
Thus,
\begin{equation}\label{eq7715}
    \tau\leq \Vert v^\pm\Vert,\ \text{where}\ \tau=\left(\frac{1}{2C_{\varepsilon}C_2}\right)^{\frac{1}{4}}.
 \end{equation}
This ends the proof of  Claim.\\
Since $\lbrace u_n\rbrace_{n\in\mathbb{N}}\subset \mathcal{M}$, by \eqref{8eq99} and \eqref{eq7715}, we have
$$\tau^2\leq \int_{\mathbb{R}^3}f(x,u_n^\pm)u_n^\pm dx,\ \ \text{for all}\ \  n\in \mathbb{N}.$$
Hence, passing to the limit as $n\rightarrow +\infty$ in the previous inequality, we see that
$$ \tau^{2}\leq \int_{\mathbb{R}^3}f(x,u^\pm)u^\pm dx.$$
Therefore, $u^\pm\not\equiv 0.$ This completes the proof.
\end{proof}

The following proposition proves that the infimum of $J$ is attained on $\mathcal{M}$:
\begin{prop}\label{4prop15}
Assume that assumptions $(f_1)-(f_2)$ and $(V_0)$ hold. Then, there exists $\widehat{w}\in\mathcal{M}$ such that $J(\widehat{w})=c_1$.
\end{prop}
\begin{proof}
First observe that from Remark \ref{rem4}, there exists a sequence $\lbrace w_n\rbrace_{n\in\mathbb{N}}\subset \mathcal{M}$ such that
$$J(w_n)\longrightarrow c_1\ \ \text{as}\ \ n\longrightarrow +\infty.$$
Arguing as in the proof of Proposition \ref{4prop5}, we deduce that $\lbrace w_n\rbrace_{n\in\mathbb{N}}$ is  bounded in $H^{1}_V(\mathbb{R}^3)$.
Thus, up to a subsequence, still denoted by $w_n$, there exists $\widehat{w}\in H^{1}_V(\mathbb{R}^3)$ such that
\begin{equation}\label{4eq22}
  \left\lbrace\begin{array}{ll}
   w_n\rightharpoonup \widehat{w}\ &\text{in}\ H^{1}_V(\mathbb{R}^3),\\
   w_n(x)\rightarrow \widehat{w}(x),\  &\text{for a.a.}\ x\in\mathbb{R}^3 \\
  \text{and} & \ \\
  w_n^\pm(x)\rightarrow \widehat{w}^\pm(x), & \text{for a.a.}\ x\in\mathbb{R}^3.
   \end{array}
  \right.
\end{equation}
According to Lemma  \ref{lem5}, we see that
$\widehat{w}^\pm \not\equiv 0$. It follows, by Lemma \ref{4lem11}, that there exists a unique pair $(t_{\widehat{w}^+},s_{\widehat{w}^-})\in (0,+\infty)\times(0,+\infty)$ such that
\begin{equation}\label{4eq99}
t_{\widehat{w}^+}\widehat{w}^{+}+s_{\widehat{w}^-}\widehat{w}^{-}\in \mathcal{M}.
\end{equation}
Namely,
\begin{equation}\label{4eq13}
\langle J^{'}(t_{\widehat{w}^+}\widehat{w}^{+}+s_{\widehat{w}^-}\widehat{w}^{-}),\widehat{w}^+\rangle=0= \langle J^{'}(t_{\widehat{w}^+}\widehat{w}^{+}+s_{\widehat{w}^-}\widehat{w}^{-}),\widehat{w}^-\rangle.
\end{equation}
Since $\lbrace w_n\rbrace_{n\in\mathbb{N}}\subset \mathcal{M}$, using \eqref{4eq22}, \eqref{4eq99}, Lemma \ref{4lem16}, and Fatou's lemma, we obtain
\begin{align}\label{4eq23}
c_1&=\lim_{n\rightarrow +\infty}J(w_n)\\
& \geq \liminf_{n\rightarrow +\infty}J(t_{\widehat{w}^+}w^+_n+s_{\widehat{w}^-}w^-_n)\nonumber\\
& \geq J(t_{\widehat{w}^+}\widehat{w}^++s_{\widehat{w}^-}\widehat{w}^-)\nonumber\\
& \geq c_1.
\end{align}
Therefore,
\begin{equation}\label{4eq24}
 c_1=\inf\limits_{\mathcal{M}}J=J(t_{\widehat{w}^+}\widehat{w}^++s_{\widehat{w}^-}\widehat{w}^-).
\end{equation}
 \textbf{Claim:} $t_{\widehat{w}^+}=s_{\widehat{w}^-}=1$. Indeed, we divide the proof of claim into two steps.\\
 \textbf{Step 1:} $0<t_{\widehat{w}^+},s_{\widehat{w}^-}\leq 1$. In fact, using \eqref{4eq22}, Lemma \ref{lem1}, and Fatou's lemma, we find that
 \begin{align}\label{4eq150}
\int_{\mathbb{R}^3}\vert\nabla \widehat{w}^{\pm}\vert^2 dx +\int_{\mathbb{R}^3} V(x)(\widehat{w}^{\pm})^2dx+q^2\int_{\mathbb{R}^3} \phi_{\widehat{w}} (\widehat{w}^{\pm})^2dx\leq \int_{\mathbb{R}^3}f(x,\widehat{w}^\pm)\widehat{w}^\pm dx.
 \end{align}
By \eqref{4eq13} and Lemma \ref{lem1}, we infer
 \begin{align}\label{4eq200}
\int_{\mathbb{R}^{3}}f(x,t_{\widehat{w}^+}\widehat{w}^{+})t_{\widehat{w}^+}\widehat{w}^{+}\ dx & =t_{\widehat{w}^+}^2\left[\int_{\mathbb{R}^3}\vert\nabla \widehat{w}^{+}\vert^2 dx +\int_{\mathbb{R}^3} V(x)(\widehat{w}^{+})^2dx\right]\nonumber\\
&+q^2t_{\widehat{w}^+}^4\int_{\mathbb{R}^3} \phi_{\widehat{w}^+}(x) (\widehat{w}^{+})^2dx+q^2t_{\widehat{w}^+}^2s_{\widehat{w}^-}^2\int_{\mathbb{R}^3} \phi_{\widehat{w}^-} (x)(\widehat{w}^{+})^2dx.
 \end{align}
Without loss of generality, we may assume that $t_{\widehat{w}^+}\geq s_{\widehat{w}^-}$. By \eqref{4eq200},  it yields that
 \begin{align}\label{8eq55}
\int_{\mathbb{R}^{3}}\frac{f(x,t_{\widehat{w}^+}\widehat{w}^{+})}{\vert t_{\widehat{w}^+}\widehat{w}^{+}\vert^3}\vert\widehat{w}^{+}\vert^3\widehat{w}^{+}\ dx
  &\leq \frac{1}{t_{\widehat{w}^+}^2}\left[\int_{\mathbb{R}^3}\vert\nabla \widehat{w}^{+}\vert^2 dx +\int_{\mathbb{R}^3} V(x)(\widehat{w}^{+})^2dx\right]\nonumber\\
&+q^2\int_{\mathbb{R}^3} \phi_{\widehat{w}} (x)(\widehat{w}^{+})^2dx.
 \end{align}
Arguing by contradiction,  suppose that $t_{\widehat{w}^+}>1$.
 Subtracting  \eqref{4eq150} from \eqref{8eq55}, and using \eqref{1}, we obtain
 \begin{align*}
0  & \leq\int_{\mathbb{R}^{d}}\left[ \frac{f(x,t_{\widehat{w}^+}\widehat{w}^{+})}{ \vert t_{\widehat{w}^+}\widehat{w}^{+}\vert^{3}}-\frac{f(x,\widehat{w}^{+})}{ \vert\widehat{w}^{+}\vert^{3}}\right]\vert \widehat{w}^+\vert^3 \widehat{w}^+dx\\
& \leq \left(\frac{1}{t_{\widehat{w}^+}^2}-1\right)\left[\int_{\mathbb{R}^3}\vert\nabla \widehat{w}^{+}\vert^2 dx +\int_{\mathbb{R}^3} V(x)(\widehat{w}^{+})^2dx\right]\\
& <0.
 \end{align*}
Thus,  a contradiction holds. Therefore, $0<t_{\widehat{w}^+},s_{\widehat{w}^-} \leq 1.$\\
\textbf{Step 2:} $t_{\widehat{w}^+}=s_{\widehat{w}^-}=1$. Indeed, we argue by contradiction, suppose that $(t_{\widehat{w}^+},s_{\widehat{w}^-})\neq (1,1).$\\
It follows, By \eqref{4eq99}, Lemma \ref{9lem1},  and Fatou's lemma, that
\begin{align*}
c_1 &\leq J(t_{\widehat{w}^+}\widehat{w}^++s_{\widehat{w}^-}\widehat{w}^-)\nonumber\\ &=J(t_{\widehat{w}^+}\widehat{w}^++s_{\widehat{w}^-}\widehat{w}^-)-\frac{1}{4}\langle J^{'}(t_{\widehat{w}^+}\widehat{w}^++s_{\widehat{w}^-}\widehat{w}^-),t_{\widehat{w}^+}\widehat{w}^++s_{\widehat{w}^-}\widehat{w}^-\rangle\nonumber\\
& < J(\widehat{w})-\frac{1}{4}\langle J^{'}(\widehat{w}),\widehat{w}\rangle\nonumber\\
& \leq \liminf_{n\rightarrow +\infty}\left[  J(w_n)-\frac{1}{4}\langle J^{'}(w_n),w_n\rangle\right] \nonumber\\
&=\liminf_{n\rightarrow +\infty} J(w_n)\ \ (\text{since}\ w_n\in \mathcal{M})\nonumber\\
& = c_1.
\end{align*}
Thus,  a contradiction holds. Hence, $t_{\widehat{w}^+}=s_{\widehat{w}^-}=1$. According to \eqref{4eq24}, it comes that
$$c_1=\inf\limits_{\mathcal{M}}J=J(t_{\widehat{w}^+}\widehat{w}^++s_{\widehat{w}^-}\widehat{w}^-)=J(\widehat{w}).$$
This ends the proof.
\end{proof}
\begin{prop}\label{4prop14}
Under the assumptions $(f_1)-(f_4)$ and  $(V_0)$, $\widehat{w}$ is a critical point for the functional $J$.
\end{prop}
\begin{proof}[Proof]
We consider the functionals $\xi_\pm:H^{1}_V(\mathbb{R}^3)\rightarrow \mathbb{R}$ defined by
\begin{align*}
\xi_\pm(w)&: =\langle J^{'}(w),w^\pm\rangle \\
&=\int_{\mathbb{R}^3}\vert\nabla w^{\pm}\vert^2 dx +\int_{\mathbb{R}^3} V(x)(w^{\pm})^2dx+q^2\int_{\mathbb{R}^3} \phi_w(x) (w^{\pm})^2dx-\int_{\mathbb{R}^3}f(x,w^{\pm})w^{\pm}dx.
\end{align*}
It's clear that $\xi_\pm$  is locally Lipschitz ( see \cite[Theorem 2.7.2, p. 221]{419} ).\\
Let $w\in H^{1}_V(\mathbb{R}^3)$. For all $\xi^*_{w^\pm}\in \partial \xi_\pm(w)$, there is $f^*(x,w^\pm)\in \partial_{w^\pm} f(x,w^\pm)$ such that
\begin{align}\label{4eq400}
\langle\xi_{w^\pm}^{*},v\rangle &  =2\int_{\mathbb{R}^3}\nabla w^{\pm}\nabla v^\pm dx +2\int_{\mathbb{R}^3} V(x)w^{\pm}v^\pm dx+2q^2\int_{\mathbb{R}^3} \phi_w w^{\pm}v^\pm dx-\int_{\mathbb{R}^3}f(x,w^{\pm})v^{\pm}dx\nonumber\\
&+2q^2\int_{\mathbb{R}^3}\int_{\mathbb{R}^3}\frac{1-e^{-\frac{\vert y-x\vert}{a}}
   }{\vert y-x\vert}w(y)v(y)w^\pm(x)^2dydx -\int_{\mathbb{R}^3}f^*(x,w^{\pm})w^{\pm}v^\pm dx,
\end{align}
for all $v\in H^{1}_V(\mathbb{R}^3)$.\\
By Proposition \ref{4prop15},
$$J(\widehat{w})=c_1=\inf\left\lbrace J(w):\ w\in H^{1}_V(\mathbb{R}^3)\setminus\lbrace  0\rbrace,\ w^\pm\not\equiv 0,\ \xi_+(w)=0=\xi_-(w) \right\rbrace.$$
According to the non-smooth multiplier rule of Clarke \cite[Theorem 10.47, p. 221]{418}, there exist $\lambda_+,\lambda_-\geq 0$ such that
$$
0\in\partial(J+ \lambda_+ \xi_++\lambda_- \xi_-)(\widehat{w}).
$$
The subdifferential calculus of Clarke \cite[p. 48]{419}, gives that

$$0\in\partial J(\widehat{w})+\lambda_+ \partial\xi_+(\widehat{w})+\lambda_- \partial\xi_-(\widehat{w}).$$
Then,
\begin{equation}\label{44eq11}
0=J^{'}(\widehat{u})+\lambda_+ \xi_{w^+}^*+\lambda_- \xi_{w^-}^*\ \ \text{in}\ \left(H^{1}_V(\mathbb{R}^3)\right)^*,\ \text{for all}\ \xi^*_{\widehat{w}^+}\in \partial\xi_+(\widehat{w})\ \text{and all}\ \xi^*_{\widehat{w}^-}\in \partial\xi_-(\widehat{w}).
\end{equation}
 Since $\widehat{w}\in \mathcal{M}$,
 \begin{equation}\label{44eq10}
0=\langle J^{'}(\widehat{w}),\widehat{w}\rangle+\lambda_+\langle \xi^{*}_{\widehat{w}^+},\widehat{w}\rangle+\lambda_-\langle \xi^{*}_{\widehat{w}^-},\widehat{w}\rangle =\lambda_+\langle \xi^{*}_{\widehat{w}^+},\widehat{w}\rangle+\lambda_-\langle \xi^{*}_{\widehat{w}^-},\widehat{w}\rangle,
\end{equation}
for all $\xi^*_{\widehat{w}^+}\in \partial\xi_+(\widehat{w})$ and all $ \xi^*_{\widehat{w}^-}\in \partial\xi_-(\widehat{w}).$\\
Using \eqref{4eq400} and the fact that $\widehat{w}\in\mathcal{M}$, we obtain
\begin{align}\label{4eq27}
\langle\xi_{\widehat{w}^\pm}^{*},\widehat{w}\rangle &  =2\int_{\mathbb{R}^3}\vert\nabla \widehat{w}^{\pm}\vert^2 dx +2\int_{\mathbb{R}^3} V(x)(\widehat{w}^{\pm})^2 dx+4q^2\int_{\mathbb{R}^3} \phi_{\widehat{w}}(x) (\widehat{w}^{\pm})^2 dx\nonumber\\
&-\int_{\mathbb{R}^3}f(x,\widehat{w}^{\pm})\widehat{w}^{\pm}dx-\int_{\mathbb{R}^3}f^*(x,\widehat{w}^{\pm})(\widehat{w}^{\pm})^2dx\nonumber\\
& =-2\Vert \widehat{w}^\pm\Vert^2+\int_{\mathbb{R}^{3}} 3f(x,\widehat{w}^{\pm})\widehat{w}^{\pm}- f^{*}(x,\widehat{w}^{\pm})(\widehat{w}^{\pm})^{2}dx.
\end{align}
It follows, by assumption $(f_4)$, that
\begin{equation}\label{4eq28}
\langle \xi^{*}_{\widehat{w}^+},\widehat{w}\rangle < 0\ \ \text{and}\ \ \langle \xi^{*}_{\widehat{w}^-},\widehat{w}\rangle <0.
\end{equation}
In light of \eqref{44eq10}, it comes that $\lambda_\pm=0.$ Therefore, from \eqref{44eq11}, we conclude that $\widehat{w}$  is a critical point of the functional $J$.\\
This completes the proof.
\end{proof}
\begin{prop}\label{4prop30}
Assume that assumptions $(f_1)-(f_4)$ and  $(V_0)$  hold. Then, the ground state solution $\widehat{u}$ of problem \eqref{P} with  a fixed sign. Moreover, $$2c_0=2J(\widehat{u})=2\inf\limits_{\mathcal{N}}J<\inf\limits_{\mathcal{M}}J=J(\widehat{w})=c_1.$$
\end{prop}
\begin{proof}
We argue by contradiction. Suppose that $\widehat{u}^\pm\not\equiv 0$, then
\begin{equation}\label{4eq33}
c_0=\inf\limits_{\mathcal{N}}J\geq \inf\limits_{\mathcal{M}}J=c_1.
\end{equation}
Since $\mathcal{M}\subset \mathcal{N}$,
\begin{equation}\label{4eq32}
c_0=\inf\limits_{\mathcal{N}}J\leq \inf\limits_{\mathcal{M}}J=c_1.
\end{equation}
Combining \eqref{4eq33} and  \eqref{4eq32}, we get
\begin{equation}\label{4eq34}
c_0=J(\widehat{u})=\inf\limits_{\mathcal{N}}J= \inf\limits_{\mathcal{M}}J=J(\widehat{w})=c_1.
\end{equation}
On the other hand, since $\widehat{w}\in\mathcal{M}$, $\widehat{w}^\pm\not\equiv 0$. Then, from Lemma  \ref{lem3}, there is a unique pair $t_{\widehat{w}^+},s_{\widehat{w}^-}>0$ such that
 $$t_{\widehat{w}^+}\widehat{w}^+\in \mathcal{N}\ \text{and}\ s_{\widehat{w}^-}\widehat{w}^-\in \mathcal{N}.$$
By Lemma  \ref{4lem16}, it follows that
\begin{align}\label{4eq31}
2c_0 &\leq J(t_{\widehat{w}^+}\widehat{w}^+)+J(s_{\widehat{w}^-}\widehat{w}^-)<J(t_{\widehat{w}^+}\widehat{w}^++s_{\widehat{w}^-}\widehat{w}^-)\leq J(\widehat{w})=\inf\limits_{\mathcal{M}}J=c_1.
\end{align}
Which is a contradiction with \eqref{4eq34}. Therefore, $\widehat{u}$ has a fixed sign, and
$$2c_0=2J(\widehat{u})=2\inf\limits_{\mathcal{N}}J<\inf\limits_{\mathcal{M}}J=J(\widehat{w})=c_1.$$
Thus the proof.
\end{proof}
\begin{proof}[\textbf{Proof of Theorem \ref{thm2} :}]  Theorem \ref{thm2} deduced from the Propositions \ref{4prop14} and \ref{4prop30}.
\end{proof}

\end{document}